\documentclass{birkjour}
\usepackage{amsmath}
\usepackage{amssymb}
\usepackage{amsfonts}
\usepackage{float}
\usepackage{graphicx}
\usepackage{cite}
\usepackage{color}
\usepackage{epstopdf}
\usepackage [latin1]{inputenc}
\usepackage{mathrsfs}
\usepackage[noend]{algpseudocode}
\usepackage{algorithmicx,algorithm}
\usepackage{lineno}
\usepackage{listings}
\usepackage{color}
\usepackage{xcolor}
\definecolor{dkgreen}{rgb}{0,0.6,0}
\definecolor{gray}{rgb}{0.5,0.5,0.5}
\definecolor{mauve}{rgb}{0.58,0,0.82}
 \newtheorem{thm}{Theorem}[section]
 \newtheorem{cor}{Corollary}[section]
 \newtheorem{lem}{Lemma}[section]
 \newtheorem{prop}{Proposition}[section]
 \newtheorem{rem}{Remark}
 \theoremstyle{definition}
 
 \theoremstyle{remark}

 \numberwithin{equation}{section}

\begin{document}

\title[Algorithm for monochromatic vertex-disconnection of graphs]
{Algorithm for monochromatic vertex-\\  disconnection of graphs}

\author[M. Fu]{Miao Fu*}
\address{%
Center for Applied Mathematics \\
Tianjin University\\
Tianjin, 300354\\
P. R. China}
\email{fumiao1119@163.com}

\author[Y.Q. Zhang ]{Yuqin Zhang}
\address{%
School of Mathematics \\
Tianjin University\\
Tianjin, 300072\\
P. R. China}
\email{yuqinzhang@tju.edu.cn}

\thanks{*Miao Fu: fumiao1119@163.com
\\Center for Applied Mathematics, Tianjin University, Tianjin, P. R. China 300354
\\Yuqin Zhang: yuqinzhang@tju.edu.cn
\\School of Mathematics, Tianjin University, Tianjin, P. R. China 300072
}



\date{}

\begin{abstract}
  Let $G$ be a vertex-colored graph. We call a vertex cut $S$ of $G$ a $monochromatic$ $vertex$ $cut$ if the vertices of $S$ are colored with the same color. The graph $G$ is $monochromatically$ $vertex$-$disconnected$ if any two nonadjacent vertices of $G$ has a monochromatic vertex cut separating them. The $monochromatic$ $vertex$-$disconnection$ $number$ of $G$, denoted by $mvd(G)$, is the maximum number of colors that are used to make $G$ monochromatically vertex-disconnected.
  In this paper, we propose an algorithm to compute $mvd(G)$ and give an $mvd$-coloring of the graph $G$. We run this algorithm with an example written in Java. The main part of the code is shown in Appendix $B$ and the complete code is given on Github: https://github.com/fumiaoT/mvd-coloring.git.
  Secondly, inspired by the previous localization principle, we obtain a upper bound of $mvd(G)$ for some special classes of graphs.
  In addition, when $mvd(G)$ is large and all blocks of $G$ are minimally $2$-connected triangle-free graphs, we characterize $G$.
  On these bases, we show that any graph whose blocks are all minimally $2$-connected graphs of small order, can be computed $mvd(G)$ and given an $mvd$-coloring in polynomial time.
\end{abstract}
\maketitle

\section{ Introduction }
\quad \quad All graphs considered in this paper are connected, finite, undirected. We follow the terminology and notation of Bondy and Murty \cite{ref1}. Let $G =\big(V(G)$, $E(G)\big)$ be a nontrivial connected graph with vertex set $V(G)$ and edge set $E(G)$. Let $|G|$ be the $order$ of $G$. For $D\subseteq E(G)$, $G-D$ is the graph obtained by removing $D$ from $G$. For $S\subseteq V(G)$, $G-S$ is the graph obtained by removing $S$ and the edges incident to the vertices of $S$ from $G$, and $G[S]$ is the subgraph of $G$ $induced$ by $S$. A graph $G$ is $minimally$ $k$-$connected$ ($minimal\ block$) if $G$ is $k$-connected but $G-\{e\}$ is not $k$-connected for every $e\in E(G)$. A $block$ of a graph $G$ is a maximal connected subgraph of $G$ that has no cut-vertex. So every block of a nontrivial connected graph is either $K_{2}$ or a $2$-connected subgraph, which is called $trivial$ and $nontrivial$, respectively. All the blocks of a graph $G$ form a $block\ decomposition$ of $G$.

An $edge$-$coloring$ of a graph $G$ is a function $\Gamma: E(G)\rightarrow[k]$ for some positive integer $k$, where adjacent edges may be assigned the same color. A $vertex$-$coloring$ of a graph $G$ is a function $\tau: V(G)\rightarrow[k]$ for some positive integer $k$, where adjacent vertices may be assigned the same color.

The notion rainbow connection coloring was introduced by Chartrand et al. \cite{ref2} in 2008. A $rainbow$ $connection$ $coloring$ of a graph $G$ is an edge-coloring of $G$ such that any two distinct vertices of $G$ are connected by a rainbow path (a path of $G$ whose edges are colored pairwise differently). The notion $rainbow$ $vertex$-$connection$ $coloring$ was introduced by Krivelevich and Yuster \cite{ref9} in 2010, which is defined from the vertex-version. For more details about the rainbow connection coloring and vertex-connection coloring, we refer to \cite{ref10, ref11, ref12, ref13, ref20}.

In contrast to the concepts of rainbow connection coloring and vertex-connection coloring, monochromatic versions of these concepts appear naturally as the other extremal. The notion monochromatic connection coloring was introduced by Caro and Yuster \cite{ref3} in 2011. A $monochromatic$ $connection$ $coloring$ of a graph $G$ is an edge-coloring of $G$ such that any two distinct vertices are connected by a monochromatic path (a path of $G$ whose edges are colored the same). The notion $monochromatic$ $vertex$-$connection$ $coloring$ was introduced by Cai, Li and Wu \cite{ref5} in 2018, which is defined from the vertex-version. For more results on the monochromatic connection coloring and vertex-connection coloring, we refer to \cite{ref4, ref6, ref8, ref14}.

There are two ways to study the connectivity of graphs, one using paths and the other using cuts. These concepts mentioned above use paths, so it is natural to consider monochromatic edge cuts and monochromatic vertex cuts in colored graphs. The notion monochromatic disconnection coloring was introduced by Li and Li \cite{ref16} in 2021. An edge-colored graph $G$ is $monochromatically$ $disconnected$ if any two distinct vertices of $G$ has a monochromatic edge cut (an edge cut of $G$ whose edges are colored the same) separating them. An edge-coloring of $G$ is a $monochromatic$ $disconnection$ $coloring$ if it makes $G$ monochromatically disconnected. For more results, we refer to \cite{ref17, ref18}.

In order to study the monochromatic vertex cuts, we now introduce the notion of monochromatic vertex-disconnection coloring in this paper. For a vertex-colored graph $G$, we call a vertex cut $S$ a $monochromatic$ $vertex$ $cut$ if the vertices of $S$ are colored with the same color. For two distinct vertices $x, y$ of $G$, a $monochromatic$ $x$-$y$ $vertex$ $cut$ is a monochromatic vertex cut that separates $x$ and $y$. Obviously, if $x$ is adjacent to $y$, there is no $x$-$y$ vertex cut, so we only need to consider nonadjacent vertices in the following. A vertex-colored graph $G$ is $monochromatically$ $vertex$-$disconnected$ if any two nonadjacent vertices of $G$ has a monochromatic vertex cut separating them. A vertex-coloring of $G$ is a $monochromatic$ $vertex$-$disconnection$ $coloring$ ($MVD$-$coloring$ for short) if it makes $G$ monochromatically vertex-disconnected. For a vertex-colored graph $G$, the $monochromatic$ $vertex$-$disconnection$ $number$ of $G$, denoted by $mvd(G)$, is the maximum number of colors that are used to make $G$ monochromatically vertex-disconnected. A monochromatic vertex-disconnection coloring with $mvd(G)$ colors is called an $mvd$-$coloring$ of $G$. Naturally, for complete graphs $K_{n}$, we define $mvd(K_{n})=n$. For a vertex-coloring $\tau$ of $G$, $\tau(v)$ is the color of vertex $v$, $\tau(G)$ is the set of colors used in $G$, and $|\tau(G)|$ is as the number of colors in $\tau(G)$. If $\tau$ is a vertex-coloring of $G$ and $H$ is a subgraph of $G$, then the part of the coloring of $\tau$ on $H$ is called $\tau$ $restricted\ to$ $H$.

Suppose $G$ is a graph that may have multiple edges but no loops, and $G'$ is a simple graph corresponding to it. Two vertices $x$ and $y$ are adjacent in $G$ if and only if they are also adjacent in $G'$. Therefore, we have the following result, which means that we only need to consider simple graphs in this paper.

\begin{prop}\label{underling}
If $G'$ is a simple graph corresponding to a graph $G$, then $mvd(G')=mvd(G)$.
\end{prop}

The monochromatic vertex-disconnection number is not only a natural combinatorial measure, but also applied in logistics transportation, transportation, computer network and many other fields. For example, in the process of logistics transportation, we often need to intercept threatening goods, such as smuggled drugs and confidential documents. We intercept the goods in the cities it may pass through. If the interception is successful in a city, that city can give feedback on the success of the interception by transmitting a signal (e.g. some special frequency) to other cities. Conversely, if no feedback is received, a new round of interceptions will be required.
This situation can be abstracted as a graphical model. Consider cities as vertices, and if there is a transport road between two cities, we assign an edge between these two vertices.
The resulting graph is denoted by $G$ and each vertex is assigned a color according to the frequency transmitted by the city. Suppose the goods are transported from city $x$ to city $y$. If there is a transport road between city $x$ and city $y$ (vertices $x$ and $y$ are adjacent), we can obviously intercept the goods directly in city $y$. If vertices $x$ and $y$ are nonadjacent, in order to intercept the goods, we need to consider the $x$-$y$ vertex cut. In order to determine whether the interception is successful, the $x$-$y$ vertex cut in $G$ needs to be monochromatic, in other words, the cities corresponding to the $x$-$y$ vertex cut have the same frequency. If a city (corresponding to a vertex in the $x$-$y$ vertex cut) transmits this frequency, it indicates a successful interception. Conversely, a new round of interceptions is required for other cities. Then the minimum number of frequencies for cities required in this problem is precisely the monochromatic vertex-disconnection number of the corresponding graph.

The rest of this paper is organized as follows. In Section $2$, we give an algorithm to compute $mvd(G)$ and get an $mvd$-coloring of $G$. This algorithm transforms the global problem of computing $mvd(G)$ and the $mvd$-coloring of $G$ into a local problem for each block of $G$, which greatly simplifies the original problem. In Section $3$, inspired by the localization principle in the previous section, we obtain a upper bound of $mvd(G)$ for some special classes of graphs. In addition, when $mvd(G)$ is large and all blocks of $G$ are minimally $2$-connected triangle-free graphs, we characterize $G$. On these bases, we show that any graph whose blocks are all minimally $2$-connected graphs of small order, can be computed $mvd(G)$ and given an $mvd$-coloring in polynomial time.

\section{ Algorithm for $mvd$-coloring }
\quad \quad In this section, we study the correlation between the $mvd$-coloring of a graph $G$ and the $mvd$-coloring of its blocks, based on which we write an algorithm to compute $mvd(G)$ and give an $mvd$-coloring of $G$: construct a $type\ set$ consisting of all graphs whose $mvd$-coloring is known; input a graph $G$ and the algorithm performs a block decomposition of $G$ (Algorithm \ref{algorithm2}); find graphs in the type set that are isomorphic to these blocks and color these blocks according to the colors in the type set, then complete the $mvd$-coloring of $G$ (Algorithm \ref{algorithm3}). Finally, we run Algorithm \ref{algorithm2} and Algorithm \ref{algorithm3} with an example written in Java.

\subsection{ Block decomposition algorithm }
\quad \quad Our algorithm is based on the block decomposition algorithm of the undirected graph $G$. In 1972, Tarjan \cite{ref22} showed how block decomposition problem can be solved efficiently by means of $Depth$-$first\ search\ (DFS)$. However, Tarjan's algorithm is skillful and the part of this algorithm that solves connected components of undirected graphs is less known. Thus in this section, first, we re-state and prove the principle of the undirected graph block decomposition algorithm in a more understandable way than in \cite{ref22}. Subsequently, we combine this principle with the variables introduced in Tarjan's algorithm to obtain the pseudo-code for block decomposition.

\subsubsection{ The principle of block decomposition algorithm }
\paragraph{} \quad \quad In DFS, we always start exploring from the most recently discovered vertex with unexplored edges and maintain a vertex, called ACTIVE. DFS number the vertices from $1$ to $n$ in the order when they are discovered. $dfs(v)$ is defined as the number assigned to vertex $v$. For each vertex $v$ apart from $x$ (the root) we can record the vertex $parent(v)$ from which $v$ has been discovered. Then $parent(v)$ is called the $parent$ of $v$. A $rooted\ tree$ is a directed tree with one vertex $x$ chosen as root. The rooted tree that we obtain by applying DFS on an undirected graph $G$ is called a $DFS\ tree$. In a DFS tree, vertex $w$ is called an $ancestor$ of $v$, and $v$ is called a $descendant$ of $w$ if there is a directed path from $w$ to $v$. The edges of a DFS tree are called $tree\ edges$ and consist of all directed edges of the form $parent(v)\longrightarrow v$. All other edges of the graph $G$ are called $back\ edges$, and obviously each back edge connects some vertex $v$ to one of its ancestors.

Now we state the principle of the block decomposition algorithm and provide a rigorous theoretical proof.

\begin{thm}\label{block1}
$G$ is a connected graph. Maintain a vertex called ACTIVE. Build a DFS tree $T$ of $G$, and delete parts of $T$ as blocks are identified. We can get the blocks of $G$ as follows.
\makeatletter
\newenvironment{algorithm2.1}
  {
   \begin{center}
     \refstepcounter{algorithm}
     \hrule height.8pt depth0pt \kern2pt
     \renewcommand{\caption}[2][\relax]{
       {\raggedright\textbf{\ALG@name~\thealgorithm} ##2\par}
       \ifx\relax##1\relax
         \addcontentsline{loa}{algorithm}{\protect\numberline{\thealgorithm}##2}
       \else
         \addcontentsline{loa}{algorithm}{\protect\numberline{\thealgorithm}##1}
       \fi
       \kern2pt\hrule\kern2pt
     }
  }{
     \kern2pt\hrule\relax
   \end{center}
  }
\makeatother

\begin{algorithm2.1}
\begin{algorithmic}[1]
\State {\bf Initialization:} pick a root $x \in V(G)$; make $x$ ACTIVE; set $T=\{x\}$
\While{let $v$ denote the current ACTIVE vertex; $v$ has an unexplored incident edge or $parent(v) \neq$ null ($v\neq x$)}
\If{$v$ has an unexplored incident edge $vw$}
    \State mark $vw$ ``explored''
    \If{$w\notin V(T)$}
        \State $v\longrightarrow w$ is a tree edge, add $v\longrightarrow w$ to $T$, make $w$ ACTIVE
    \Else
        \State $v\longrightarrow w$ is a back edge, $w$ is an ancestor of $v$
    \EndIf
\Else
    \State make $parent(v)$ ACTIVE
    \If{no vertex in the current subtree $T'$ rooted at $v$ has a back edge to an ancestor above $parent(v)$}
        \State $V(T')\cup \{parent(v)\}$ is the vertex set of a block; record this information and delete $V(T')$ from $T$
    \EndIf
\EndIf
\EndWhile
\end{algorithmic}
\end{algorithm2.1}
\end{thm}
\begin{proof}
The proof proceeds by induction on $|G|$. For $K_2$, the algorithm correctly identifies the single block ($K_1$ is a special case). By induction, for a larger graph $G$, it suffices to show that the first set identified as a block is indeed a block $B$ that share one vertex $parent(v)$ with the rest of $G$ (since the subsequent run of the algorithm on $G$ is equivalent to running the algorithm on the graph $G-(B-parent(v))$).

In line $10$ and $11$, we backtrack from $v$ to $parent(v)$. That is, when the ACTIVE vertex changes from $v$ to $parent(v)$, we need to check if any vertex in the subtree $T'$ rooted at $v$ has a neighbor above $parent(v)$. This is easy to do because in Line $8$, when we mark a back edge to an ancestor as explored, we record for the vertices on the path in $T$ between the endpoints of the back edge that there is a back edge from a descendant to an ancestor. When $parent(v)$ becomes ACTIVE again, we check if it was ever so marked.

We claim that $G[V(T')\cup\{parent(v)\}]$ is a block when the ACTIVE vertex changes from $v$ to $parent(v)$ and no vertex in the subtree $T'$ rooted at $v$ has a neighbor above $parent(v)$. In this case, for any vertex $j$ in $V(T')$, there is a back edge $k\longrightarrow i$ from a descendant $k$ of $j$ to an ancestor $i$ of $j$, then $i$ via some tree edges down through $j$ and finally reaches $k$, forming a cycle. Thus there is no cut-vertex in $V(T')$, in other words, no proper subset of set $V(T')$ induces a block. On the other hand, since there is no edge joining $T'$ to an ancestor of $parent(v)$, then $parent(v)$ is a cut-vertex. Therefore, $G[V(T')\cup \{parent(v)\}]$ is a maximal subgraph having no cut-vertex.
\end{proof}

\subsubsection{ Pseudo-code for block decomposition algorithm}
\paragraph{} \quad \quad We introduce a parameter $low(\cdot)$ from Tarjan's algorithm to implement line 8 and 11 in Theorem \ref{block1}. If $v$ is a leaf of the DFS tree, let $low(v)=dfs(v)$ immediately when $v$ is discovered, and update $low(v)$ to $low(v)$=min$\{low(v), dfs(w)\}$ as each back edge $v\longrightarrow w$ is explored. If $v$ is not a leaf of the DFS tree, when we backtrack from $v$ to $parent(v)$, it is clear that we have backtracked from all its descendants earlier and therefore already know their $low(\cdot)$. So in addition to what we do for the leaves, it is sufficient to do the following: when we backtrack from $v$ to $parent(v)$, update $low(parent(v))$ to $low(parent(v))$=min$\{low(parent(v)), low(v)\}$.

\noindent$\displaystyle \textbf{Example 2.1.}$ The graph shown in Fig. \ref{1}(1). We apply DFS to this graph, starting with root $x=A$, and assume that we discover vertices in the order of $A,\ B\ ,C\ ,D\ ,E\ ,F\ ,G$. The resulting $\big(dfs(\cdot),\ low(\cdot)\big)$, tree edges, and back edges are shown in Fig. \ref{1}(2).

\begin{figure}
\centering
  \includegraphics[height=4cm]{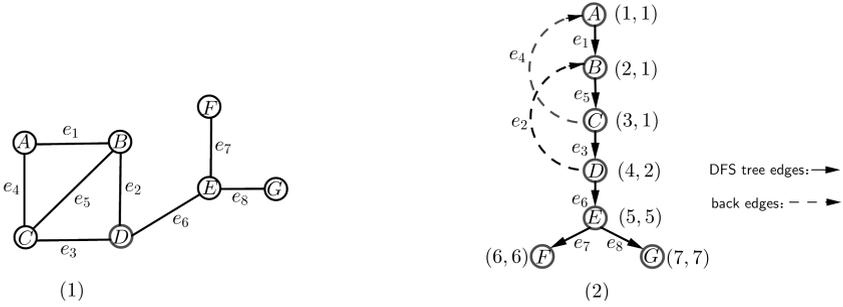}
  \caption{$\big(dfs(\cdot),\ low(\cdot)\big)$, the DFS tree edges and back edges}
  \label{1}
\end{figure}

Combining Theorem \ref{block1} with the parameter $low(\cdot)$ described above, it is easy to obtain the following corollary.

\begin{cor}\label{gedian1}
Let $G$ be a graph whose vertices have been numbered by DFS, starting with $x\ (dfs(x)=1)$. $x$ is a cut-vertex if and only if there are at least two tree edges out of $x$. If $parent(v)\longrightarrow v$ is a tree edge and $dfs(parent(v))$ $>1$, then $parent(v)$ is a cut-vertex of $G$ if and only if $low(v)\geq dfs(parent(v))$.
\end{cor}

It is worth mentioning that in DFS, we maintain the list of vertices to be searched as a stack (pushdown store). We store the vertices in the stack in the order of when they are discovered. When backtracking from $v$ to $parent(v)$, we discover that $parent(v)$ is a cut-vertex, and we read and pop all vertices from the top of the stack down to and including $v$. All these vertices, plus $parent(v)$, (which is not popped at this point from the stack even if it is the next on top) constitute a block.

Now we implement Theorem \ref{block1} and obtain Algorithm \ref{algorithm2}. The algorithm produces the set of cut-vertices of the graph $G$ as well as the set of blocks, both of which are initialized to null. $S$ is a stack of vertices.

\makeatletter
\newenvironment{breakable1algorithm}
  {
   \begin{center}
     \refstepcounter{algorithm}
     \hrule height.8pt depth0pt \kern2pt
     \renewcommand{\caption}[2][\relax]{
       {\raggedright\textbf{\ALG@name~\thealgorithm} ##2\par}
       \ifx\relax##1\relax
         \addcontentsline{loa}{algorithm}{\protect\numberline{\thealgorithm}##2}
       \else
         \addcontentsline{loa}{algorithm}{\protect\numberline{\thealgorithm}##1}
       \fi
       \kern2pt\hrule\kern2pt
     }
  }{
     \kern2pt\hrule\relax
   \end{center}
  }
\makeatother

\begin{breakable1algorithm}\label{algorithm2}
\caption{Block decomposition algorithm for graph $G$}
\hspace*{0.02in} {\bf Input:}
$(V(G),\ E(G))$, a root $x$ of $G$, current ACTIVE vertex $v$\\
\hspace*{0.02in} {\bf Output:}
set $CutVerticesSet$ of cut-vertices, set $BlocksSet$ of blocks
\begin{algorithmic}[1]
\State \textbf{for every} edge $vw\in E(G)$ mark $vw$ ``unexplored''
\For{\textbf{every} vertex $w \in V(G)$}
    \State $dfs(w) \gets 0$
    \State $parent(w)$ $\gets$ null
\EndFor
\State $v \gets x$
\State $dfs(x) \gets 1$
\State $i \gets 1$
\State vacate $S$
\State push $x$ into $S$
\While{$v$ has an unexplored incident edge or $parent(v) \neq$ null}
    \If{$v$ has an unexplored incident edge $vw$}
        \State mark $vw$ ``explored''
        \If{$dfs(w)=0$}
            \State push $w$ into $S$
            \State $parent(w) \gets v$
            \State $i=i+1$
            \State $dfs(w) \gets i$
            \State $low(w) \gets i$
            \State $v \gets w$ (make $w$ ACTIVE)
        \Else{ ($w$ is explored, that is, $w$ is an ancestor of $v$)}
            \State $low(v)$ $\gets$ min$\{low(v),dfs(w)\}$
        \EndIf
    \Else{}
        \If{$low(v)\geq dfs(parent(v))$ (by Corollary \ref{gedian1})}
            \If{$parent(v)\neq x$ or $x$ has an unexplored incident edge}
                   \State add $parent(v)$ to $CutVerticesSet$
            \EndIf
            \State pop vertices from $S$ down to and including $v$
            \State the set of popped vertices, with $parent(v)$, is an element of $BlocksSet$
        \Else{ ($low(v)<dfs(parent(v))$)}
            \State $low(parent(v))$ $\gets$ min$\{low(parent(v)),low(v)\}$
        \EndIf
        $v \gets parent(v)$ (make $parent(v)$ ACTIVE)
    \EndIf
\EndWhile
\end{algorithmic}
\end{breakable1algorithm}

\subsection{ Pseudo-code for $mvd$-coloring algorithm }
\quad \quad
\begin{lem}\label{induced}
If $\tau$ is an $MVD$-coloring of $G$, then $\tau$ restricted to $G[S]$ is also an $MVD$-coloring of $G[S]$.
\end{lem}
\begin{proof}
Let the coloring of $\tau$ restricted to $G[S]$ be denoted as $\tau'$ and let $x,\ y$ be two nonadjacent vertices of $G[S]$. If $D$ is a monochromatic $x$-$y$ vertex cut of $G$, then $D'=D\cap V(G[S])$ is a monochromatic $x$-$y$ vertex cut in $G[S]$. Otherwise, if there exists an $x,y$-path $P$ in $G[S]-D'$, then $P$ is also in $G-D$, a contradiction. Thus, $\tau'$ is an $MVD$-coloring of $G[S]$.
\end{proof}

\begin{thm}\label{everyblock}
Let the connected graph $G$ have $r$ blocks $B_{1}, B_{2},$ $\ldots , B_{r}$. Then the vertex-coloring $\tau$ of $G$ is an $mvd$-coloring if and only if $\tau$ restricted to each block is an $mvd$-coloring of each block.
\end{thm}
\begin{proof}
Let \{$B_{1}, B_{2}, \ldots, B_{r}$\} be a block decomposition of $G$. If $G$ does not have cut-vertex, then the theorem obviously holds. Now let $G$ have at least one cut-vertex. First, we claim that $\tau$ is an $MVD$-coloring of $G$ if and only if $\tau$ restricted to each block is an $MVD$-coloring of each block.

Since each block is a vertex induced subgraph of $G$, the necessity is obvious by Lemma \ref{induced}. Suppose $\tau$ is a vertex-coloring of $G$ and the coloring of $\tau$ restricted to $B_{i}$ is denoted as $\tau_{i}$. Suppose $\tau_{i}$ is an $MVD$-coloring of $B_{i}$, $i\in [r]$. For any two nonadjacent vertices $x$ and $y$ in $G$, if there exists a block $B_{i}$, which contains both $x$ and $y$, then any monochromatic $x$-$y$ vertex cut in $B_{i}$ with the coloring $\tau_{i}$ is a monochromatic $x$-$y$ vertex cut in $G$. Otherwise, there exists an $x, y$-path $P\not\subset B_{i}$. Since $B_{i}$ is a maximal $2$-connected graph, $P$ must belong to some cycles of $G$. Therefore, $P\subseteq B_{i}$, a contradiction. If there is no block containing both $x$ and $y$, that is, $x$ and $y$ are in different blocks, then there is exactly one $x,\ y$ internally disjoint path, say $P'$, in $G$ and the path $P'$ contains at least one cut-vertex, say $w$. It is obvious that  vertex $w$ is a monochromatic $x$-$y$ vertex cut in $G$.

Next, we assume $\tau_{i}$ is an $mvd$-coloring of $B_{i}$, where $i\in [r]$, and if $B_{i}\cap B_{j}=w$, then $\tau_{i}(B_{i})\cap \tau_{j}(B_{j})=\tau_{i}(w)=\tau_{j}(w)$, and if $B_{i}\cap B_{j}=\varnothing$, then $\tau_{i}(B_{i})\cap \tau_{j}(B_{j})=\varnothing$. Hence, $\tau$ is an $MVD$-coloring of $G$. We claim that $\tau$ is an $mvd$-coloring of $G$. Otherwise, suppose that $f$ is an $mvd$-coloring of $G$, then $|f(G)|>|\tau(G)|$. Let the coloring of $f$ restricted to $B_{i}$ be denoted as $f_{i}$. Then for $i\in [r]$, we have that $f_{i}$ is an $MVD$-coloring of $B_{i}$ and $|f_{i}(B_{i})|\leq |\tau_{i}(B_{i})|$ which contradicts $|f(G)|>|\tau(G)|$.

On the contrary, we set $\tau$ to be an $mvd$-coloring of $G$. Then we obtain that $\tau_{i}$ is an $MVD$-coloring of $B_{i}$, where $i\in [r]$, and if $B_{i}\cap B_{j}=w$, then $\tau_{i}(B_{i})\cap \tau_{j}(B_{j})=\tau_{i}(w)=\tau_{j}(w)$, and if $B_{i}\cap B_{j}=\varnothing$, then $\tau_{i}(B_{i})\cap \tau_{j}(B_{j})=\varnothing$. We claim that $\tau_{i}$ is also an $mvd$-coloring of $B_{i}$. Otherwise, assume that $\tau_{1}$ is not an $mvd$-coloring of $B_{1}$ and we assign some new colors to the vertices in $B_{1}$ (except for the cut-vertices). Let $g_{1}$ be the $mvd$-coloring of $B_{1}$, then $|g_{1}(B_{1})|>|\tau_{1}(B_{1})|$. Now we get a vertex-coloring $g$ of $G$, and $g$ is an $MVD$-coloring of $G$. In this case, $|g(G)|>|T(G)|$, which contradicts the maximality of $\tau$.
\end{proof}

For a graph $G$ with $r$ blocks, we have the following corollary.

\begin{cor}\label{blockrelation}
If connected graph $G$ has $r$ blocks $B_{1}, B_{2},$ $\ldots , B_{r}$, then $mvd(G)=\big(\sum\limits_{i=1}^{r} mvd(B_{i})\big)-r+1$.
\end{cor}

Theorem\ref{everyblock} and Corollary \ref{blockrelation} guarantee the existence of Algorithm \ref{algorithm3}, which is also important in the subsequent study of $mvd$-coloring. It is worth mentioning that in line 4, we need to find a graph $thBlock$ in the type set that is isomorphic to $Block$, and color $Block$ with the color of $thBlock$ according to the correspondence of their vertices (each time coloring a different block with new colors). The isomorphism algorithm is not the focus of this paper, so we do not elaborate on it. Since cut-vertices may be colored multiple times, in lines 6-16 we update the colors of the cut-vertices and finally give the $mvd$-coloring of $G$.

\makeatletter
\newenvironment{breakable2algorithm}
  {
   \begin{center}
     \refstepcounter{algorithm}
     \hrule height.8pt depth0pt \kern2pt
     \renewcommand{\caption}[2][\relax]{
       {\raggedright\textbf{\ALG@name~\thealgorithm} ##2\par}
       \ifx\relax##1\relax
         \addcontentsline{loa}{algorithm}{\protect\numberline{\thealgorithm}##2}
       \else
         \addcontentsline{loa}{algorithm}{\protect\numberline{\thealgorithm}##1}
       \fi
       \kern2pt\hrule\kern2pt
     }
  }{
     \kern2pt\hrule\relax
   \end{center}
  }
\makeatother

\begin{breakable2algorithm}\label{algorithm3}
\caption{Compute $mvd(G)$ and give an $mvd$-coloring of $G$}
\hspace*{0.02in} {\bf Input:}
$CutVerticesSet$ and $BlocksSet$ from Algorithm $2$\\
\hspace*{0.02in} {\bf Output:}
$mvd(G)$, $mvd$-coloring of $G$
\begin{algorithmic}[1]
\State $mvd(G) \gets 0$
\State $r \gets card(BlocksSet)$
\For{\textbf{every} block $Block \in BlocksSet$}
    \State give $Block$ an $mvd$-coloring with new colors
    \State $mvd(G) \gets mvd(G)+mvd(Block)$
\EndFor
$mvd(G) \gets mvd(G)-r+1$
\For{\textbf{every} vertex $u \in CutVerticesSet$}
    \State color($u$) $\gets$ null
\EndFor
\For{\textbf{every} block $Block \in BlocksSet$}
    \For{\textbf{every} vertex $v \in Block$}
        \If{$v \in CutVerticesSet$}
           \State $CutVertex \gets $cut-vertex in $CutVerticesSet$ corresponding to $v$
           \If{color($CutVertex$)=null}
               \State color($CutVertex$) $\gets$ color($v$)
           \Else
               \State color($v$) $\gets$ color($CutVertex$)
               \State update the colors of all vertices in $Block$ according to the $mvd$-coloring of $Block$
           \EndIf
        \EndIf
    \EndFor
\EndFor
\end{algorithmic}
\end{breakable2algorithm}

\subsection{ An algorithm example }
\quad \quad The graph is stored into the computer by inputting its adjacency matrix. We conclude this section by running Algorithm \ref{algorithm2} and Algorithm \ref{algorithm3} with an example. The main part of the code is shown in Appendix $B$ and the complete code is given on Github: https://github.com/fumiaoT/mvd-coloring.git.

\noindent$\displaystyle \textbf{Example 2.2.}$ The $mvd$-coloring of \textcircled{\small {9}} and \textcircled{\small {11}} in Appendix $A$. 9 VERTEX are known, they are elements of the type set and are stored at the $resource$ $directory$.
The graph $G$ is shown in Fig. \ref{6}$(1)$ and is stored into the computer by entering its adjacency matrix at the $main(\cdot)$ function in the $GraphContext\{\cdot\}$ class. The $DFSTarjan(\cdot)$ function in the $BlockAndCutVerticesBuilder\{\cdot\}$ class performs block decomposition on $G$. The $getIsomorphicColors(\cdot)$ function in the $IsomorphicJudger\{\cdot\}$ class looks for graphs in the type set that are isomorphic to the blocks of $G$. The $MvdColorMarker$ class performs $mvd$-coloring on $G$.

The output results are listed at the end of Appendix $B$.
The block decomposition algorithm outputs a cut-vertices set $H$ and two blocks whose vertices and their adjacency matrices are shown in $Block\ num\ 1$ and $Block\ num\ 2$ respectively (also Fig. \ref{7}$(1)(3)$).
The isomorphism algorithm looks for graphs in the type set that are isomorphic to the blocks of $G$ and outputs their correspondence at $Isomorphic$ $Relationship$ (also Fig. \ref{7}).
The $mvd$-coloring algorithm updates the colors of the cut-vertices and outputs the $mvd$-coloring of $G$ at $Coloring$ $Vertices$ $Results$ (also Fig. \ref{6}$(2)$).

\begin{figure}
\centering
  \includegraphics[height=3.5cm]{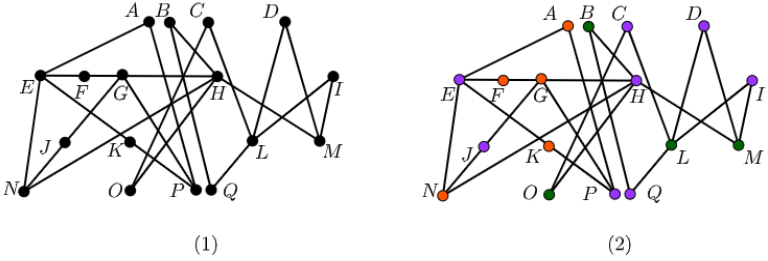}
  \caption{An example and its mvd-coloring}
  \label{6}
\end{figure}

\begin{figure}
\centering
  \includegraphics[height=3.8cm]{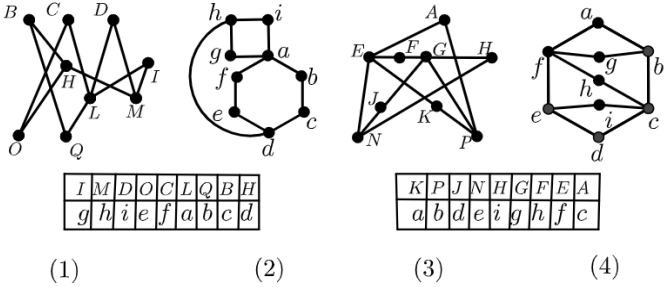}
  \caption{The output results of block decomposition and isomorphism}
  \label{7}
\end{figure}

Our algorithm can efficiently give an $mvd$-coloring of $G$ provided that for any block of $G$, there is a graph in the type set that is isomorphic to it. Otherwise, the algorithm is also able to select blocks for which no graph in the type set can be isomorphic. Meanwhile, the $mvd$-coloring of these blocks, once solved, can further enrich our type set.

For intricate graphs abstracted from real-world problems, our algorithm can perform $mvd$-coloring quickly and accurately, reducing a lot of tedious work. In the future we can make an interactive graphical interface where the user can draw the graphs on your web page and our algorithm computes all the cut-vertices and blocks based on the graphs. Also the blocks are compared with the graphs in the type set and $mvd$-coloring is performed. Finally, the cut-vertices, the colored graphs are labeled on the user's graphs. This algorithm is convenient and applicable as a mathematical tool to solve the $mvd$-coloring problem. Moreover, the algorithm can be modified in a similar way to solve a series of coloring problems, such as $rd$-coloring, $md$-coloring.

\section{ Results for special graphs }
\quad \quad Inspired by the localization principle in Theorem \ref{everyblock}, we obtain upper bounds on $mvd(G)$ for some graphs associated with minimal 2-connected graphs in this section. Moreover, when $mvd(G)$ is large and all blocks of $G$ are minimally $2$-connected triangle-free graphs, we characterize $G$.
\subsection{ The upper bound of $mvd(G)$ }
\quad \quad
\begin{lem}\label{tree}
If $G$ is a cycle of order $n\geq 4$, then $mvd(G)=\left\lfloor {\frac{n}{2}}\right\rfloor$.
\end{lem}
\begin{proof}
Let $G=v_{1}e_{1}v_{2}e_{2}\ldots v_{n-1}e_{n-1}v_{n}e_{n}v_{1}$. We define a vertex-coloring $\tau: V(G)\rightarrow \big[\left\lfloor {\frac{n}{2}}\right\rfloor\big]$ such that if $j\equiv i$ $(mod$ $\left\lfloor {\frac{n}{2}}\right\rfloor)$, then color $v_{j}$ by $i$, where $i\in\big[\left\lfloor {\frac{n}{2}}\right\rfloor\big]$, $j\in[n]$. It is easy to verify that for any two nonadjacent vertices $x$ and $y$ of $G$, there exists a monochromatic $x$-$y$ vertex cut. Therefore, $\tau$ is an $MVD$-coloring, $mvd(G)\geq \left\lfloor {\frac{n}{2}}\right\rfloor$.

For $n\geq 4$, if $mvd(G)\geq \left\lfloor {\frac{n}{2}}\right\rfloor+1$ and there is an $MVD$-coloring $\tau$ of $G$ such that $|\tau(G)|\geq \left\lfloor {\frac{n}{2}}\right\rfloor+1$, then there exists a color $i$ in $\tau$ that colors only one vertex $v_{i}$ of $G$. Otherwise, $V(G)\geq 2|\tau(G)| \geq 2(\left\lfloor {\frac{n}{2}}\right\rfloor+1)\geq n+1$. Since $G$ is a 2-connected graph, the monochromatic $v_{i-1}$-$v_{i+1}$ vertex cut must contain $v_{i}$ and some vertex in $G-\{v_{i-1}, v_{i}, v_{i+1}\}$, which contradicts the fact that $\tau$ is an $MVD$-coloring.
\end{proof}

\begin{rem}\label{remark 2}
$mvd(G)=3$, where $G$ is a $3$-cycle.
\end{rem}

A $nest\ sequence$ of graphs is a sequence $G_0, G_1, \ldots, G_t$ of graphs such that $G_i\subset G_{i+1}, 0\leq i<t$. An $ear\ decomposition$ of a $2$-connected graph $G$ is a nest sequence of subgraphs $G_0, G_1, \ldots, G_t$ of $G$ satisfying the following conditions: $(i)$ $G_0$ is a cycle of $G$; $(ii)$ $G_{i+1}=G_{i}\cup P_{i}$, where $P_{i}$ is an ear of $G_i$, $0\leq i<t$; $(iii)$ $G_{t}=G$.

\begin{thm}\label{2bound}
Let $G$ be a $2$-connected triangle-free graph and every ear $P_{i}(0\leq i< t)$ has internal vertices. Then $mvd(G)\leq \left\lfloor {\frac{n}{2}}\right\rfloor$, and the bound is sharp.
\end{thm}
\begin{proof}
Let $F=\{G_0, G_1, \ldots, G_t\}$ be an ear decomposition of $G$. We use induction on $|F|$.

Basis step: $|F|=1$, then $G$ is a cycle, the theorem holds by Lemma \ref{tree}.

Induction step: $|F|=t+1>1$, let $\tau$ be an $mvd$-coloring of $G$. Since $|P_{t-1}|\geq 3$, $G_{t-1}$ is a connected vertex induced subgraph of $G$ and $\tau$ restricted to $G_{t-1}$ is an $MVD$-coloring of $G_{t-1}$. By induction, we have
$$|\tau(G_{t-1})|\leq mvd(G_{t-1})\leq \left\lfloor {\frac{|G_{t-1}|}{2}}\right\rfloor=\left\lfloor {\frac{n-|P_{t-1}|+2}{2}}\right\rfloor.$$
Suppose that the endpoints of $P_{t-1}$ are $a$, $b$, and $L$ is the shortest $a,b$-path in $G_{t-1}$. Since $P_{t-1}$ is the last ear, cycle $C=L\cup P_{t-1}$ is a connected vertex induced subgraph of $G$. By Lemma \ref{induced}, $\tau$ restricted to $C$ is an $MVD$-coloring of $C$, therefore there are at most $|P_{t-1}|-2$ vertices are assigned colors in $\tau(G)-\tau(G_{t-1})$. Since $|C|\geq 4$, each color of $\tau(G)-\tau(G_{t-1})$ colors at least two internal vertices of $P_{t-1}$. Otherwise, if $j\in \tau(G)-\tau(G_{t-1})$ and only colors one internal vertex of $P_{t-1}$, say $x_{j}$, then $x_{j-1}, x_{j+1}$ are two nonadjacent vertices of $G$ and monochromatic $x_{j-1}$-$x_{j+1}$ vertex cut must contains $x_{j}$, a contradiction. Then $|\tau(G)-\tau(G_{t-1})|\leq \left\lfloor {\frac{|P_{t-1}|-2}{2}}\right\rfloor$. So,

\begin{equation}
\begin{aligned}
mvd(G)&=|\tau(G)|=|\tau(G_{t-1})|+|\tau(G)-\tau(G_{t-1})|\\
&\leq \left\lfloor {\frac{n-|P_{t-1}|+2}{2}}\right\rfloor+\left\lfloor {\frac{|P_{t-1}|-2}{2}}\right\rfloor\\
&\leq\left\lfloor {\frac{n}{2}}\right\rfloor.
\end{aligned}
\end{equation}

Above all,  $mvd(G)\leq \left\lfloor {\frac{n}{2}}\right\rfloor$. Furthermore, the bound is sharp for cycles of order $n\geq 4$.
\end{proof}

\begin{lem}\cite{ref15}\label{internalofPt}
Let $G$ be a minimally $2$-connected graph, and $G$ is not a cycle. Then $G$ has an ear decomposition $G_0, G_1, \ldots, G_{t}\ (t\geq 1)$ satisfying the following conditions:

$(i)$ $G_{i+1}=G_{i}\cup P_{i}\ (0\leq i<t)$, where $P_{i}$ is an ear of $G_{i}$ in $G$ and at least one vertex of $P_{i}$ has degree two in $G$,

$(ii)$ each of the two internally disjoint paths in $G_0$ between the endpoints of $P_0$ has at least one vertex with degree two in $G$.
\end{lem}

\begin{lem}\cite{ref21}\label{trianglefree}
If $G$ is a minimally $2$-connected graph of order $n\geq 4$, then $G$ contains no triangles.
\end{lem}

\begin{thm}\label{mini2bound}
If $G$ is a minimally $2$-connected graph of order $n\geq 4$, then $mvd(G)\leq \left\lfloor {\frac{n}{2}}\right\rfloor$.
\end{thm}
\begin{proof}
By Lemma \ref{internalofPt} and \ref{trianglefree}, if $G$ is not a cycle, then there is an ear decomposition satisfying the conditions in Theorem \ref{2bound}. Thus, $mvd(G)\leq \left\lfloor {\frac{n}{2}}\right\rfloor$. If $G$ is a cycle, then by Lemma \ref{tree} $mvd(G)=\left\lfloor {\frac{n}{2}}\right\rfloor$. Furthermore, the bound is sharp.
\end{proof}

Further, we obtain a upper bound of $1$-connected graph related to the number of blocks.

\begin{thm}\label{blockbound}
For a connected graph $G$ of order $n$ with $r$ blocks, $t$ of which are trivial blocks, if all blocks are minimally $2$-connected triangle-free graphs, then $mvd(G)\leq \left\lfloor {\frac{n+2t-r+1}{2}}\right\rfloor$.
\end{thm}
\begin{proof}
We first claim a connected graph $G$ with blocks $B_{1}, B_{2}, \ldots , B_{r}$ has $\big(\sum\limits_{i=1}^{r} |B_{i}|\big)-r+1$ vertices. The proof proceeds by induction on $r$. Basis step: $r=1$. A graph that is a single block $B_{1}$ has $|B_{1}|$ vertices.

Induction step: $r>1$. When $G$ is not $2$-connected, there is a block $B$ that contains only one of the cut-vertices; let this vertex be $v$, and index the blocks so that $B_r=B$. Let $G'=G-\big(V(B)-\{v\}\big)$. The graph $G'$ is connected and has blocks $B_{1}, B_{2},\ldots , B_{r-1}$. By the induction hypothesis, $|G'|=\big(\sum\limits_{i=1}^{r-1} |B_{i}|\big)-(r-1)+1$. Since we deleted $|B_{r}|-1$ vertices from $G$ to obtain $G'$, the number of vertices in $G$ is as desired.

W.l.o.g, let the trivial blocks be $B_{1}, \ldots , B_{t}$, and let the nontrivial blocks be $B_{t+1}, \ldots , B_{r}$. By Corollary \ref{blockrelation} and Theorem \ref{mini2bound},
\begin{equation}
\begin{aligned}
mvd(G)&=\Big(\sum\limits_{i=1}^{r} mvd(B_{i})\Big)-r+1\leq 2t+\left\lfloor {\frac{|B_{t+1}|}{2}}\right\rfloor+\ldots +\left\lfloor {\frac{|B_{r}|}{2}}\right\rfloor-r+1\\
&\leq \left\lfloor {\frac{2t+|B_{t+1}|+\ldots +|B_{r}|}{2}}\right\rfloor+t-r+1\\
&=\left\lfloor {\frac{n+r-1}{2}}\right\rfloor+t-r+1\\
&=\left\lfloor {\frac{n+2t-r+1}{2}}\right\rfloor.
\end{aligned}
\end{equation}

The bound is sharp for cactuses without odd cycles, where a $cactus$ is a connected graph in which every block is $K_{2}$ or a cycle.
\end{proof}

\subsection{ A contrary problem }
\quad \quad Given a graph $G$, it is meaningful to calculate $mvd(G)$. Conversely, given $mvd(G)$, the complete characterization of $G$ is also a meaningful study.
Thus in this section, in the sense of isomorphism, if blocks in $G$ are all minimal $2$-connected triangle-free graphs, we characterize $G$ when $mvd(G)$ is large. That is, Theorem \ref{large},

\begin{thm}\label{large}
For a connected graph $G$, if blocks in $G$ are all minimally $2$-connected triangle-free graphs, then
\begin{equation}
    mvd(G)=
   \begin{cases}
   n&\mbox{ $\Leftrightarrow G$ is a tree,}\\
   n-1&\mbox{ $\varnothing$,}\\
   n-2&\mbox{ $\Leftrightarrow G$ is a unicycle graph with the cycle being $C_4$,}\\
   n-3&\mbox{ $\Leftrightarrow G\in \mathscr{A}$,}\\
   n-4&\mbox{ $\Leftrightarrow G\in \mathscr{B}$,}\\
   n-5&\mbox{ $\Leftrightarrow G\in \mathscr{C}$.}
   \end{cases}
  \end{equation}
\end{thm}

Before the proof, we show some preparatory results as follows.

First, we use the notations in \cite{ref7}. Let $p_{1}, \ldots, p_{k}$ be $k$ disjoint paths such that $p_{i}$ contains $m_{i}$ vertices. $f(p_{i})$ denotes the first vertex of $p_{i}$ and $l(p_{i})$ denotes the last vertex of $p_{i}$. For vertices $u$ and $v$ that are not in $\cup_{i\in[k]}$ $V(p_{i})$, let $P(m_{1}, m_{2}, \ldots, m_{k})$ be the graph with $V(P)=\cup_{i\in[k]} V(p_{i}) \cup \{u, v\}$ and $E(P)=\cup_{i\in[k]} \big[E(p_{i}) \cup \{f(p_{i})u,l(p_{i})v\}\big]$ such that $p_{1}, \ldots, p_{k}$ are subgraphs of $P$. In the sequence $m_{1},\ldots , m_{k}$, if $m_{i+1}=m_{i+2}=\cdots=m_{i+j}$ for some integer $i$ and $j$, we can write $m_{1},\ldots , m_{k}$ in the form $m_{1},\ldots , m_{i}, j*m_{i+1}, m_{i+j+1}, \ldots ,m_{k}$.

Next, if $\tau$ is an $mvd$-coloring of $G$, $G_{i}$ is a vertex induced subgraph of $G$, and $\tau_{i}$ is the coloring of $\tau$ restricted to $G_{i}$, then it follows from Lemma \ref{induced} that $\tau_{i}$ is an $MVD$-coloring of $G_{i}$. Thus, theoretically, we can infer $\tau$ by letting each $\tau_{i}$ be an $MVD$-coloring of $G_{i}$. Taking graph $P(1,1,1)$ as an example (see Fig. \ref{3}), it follows from Lemma \ref{induced} that if $\tau$ is the $mvd$-coloring of $P(1,1,1)$, then the coloring of $\tau$ restricted to cycle $C_{abcd}$ may be (1) or (3). For case (1), the coloring of $\tau$ restricted to cycle $C_{abce}$ and cycle $C_{afcd}$ must be (2). For case (3), the coloring of $\tau$ restricted to cycle $C_{abce}$ and cycle $C_{afcd}$ must be (4). It is easy to check that both (2) and (4) are $MVD$-coloring of $P(1,1,1)$ and (4) is an $mvd$-coloring of $P(1,1,1)$. For all minimally $2$-connected graphs $G$ of order 10 or less \cite{ref7}, we use the same method and after tedious computations, we obtain Corollary \ref{6-10}, that is, $mvd(G)$. An $mvd$-coloring of $G$ is given in Appendix A.
\begin{figure}
\centering
  \includegraphics[height=2cm]{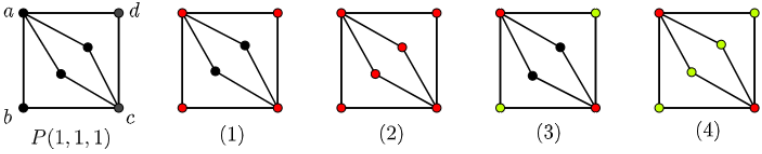}
  \caption{An example of Corollary \ref{6-10}}
  \label{3}
\end{figure}

\begin{cor}\label{6-10}
For all minimally $2$-connected graphs $G$ of order 10 or less, $mvd(G)$ is shown in the following table.
\end{cor}
\begin{figure}[H]
\centering
  \includegraphics[height=6cm]{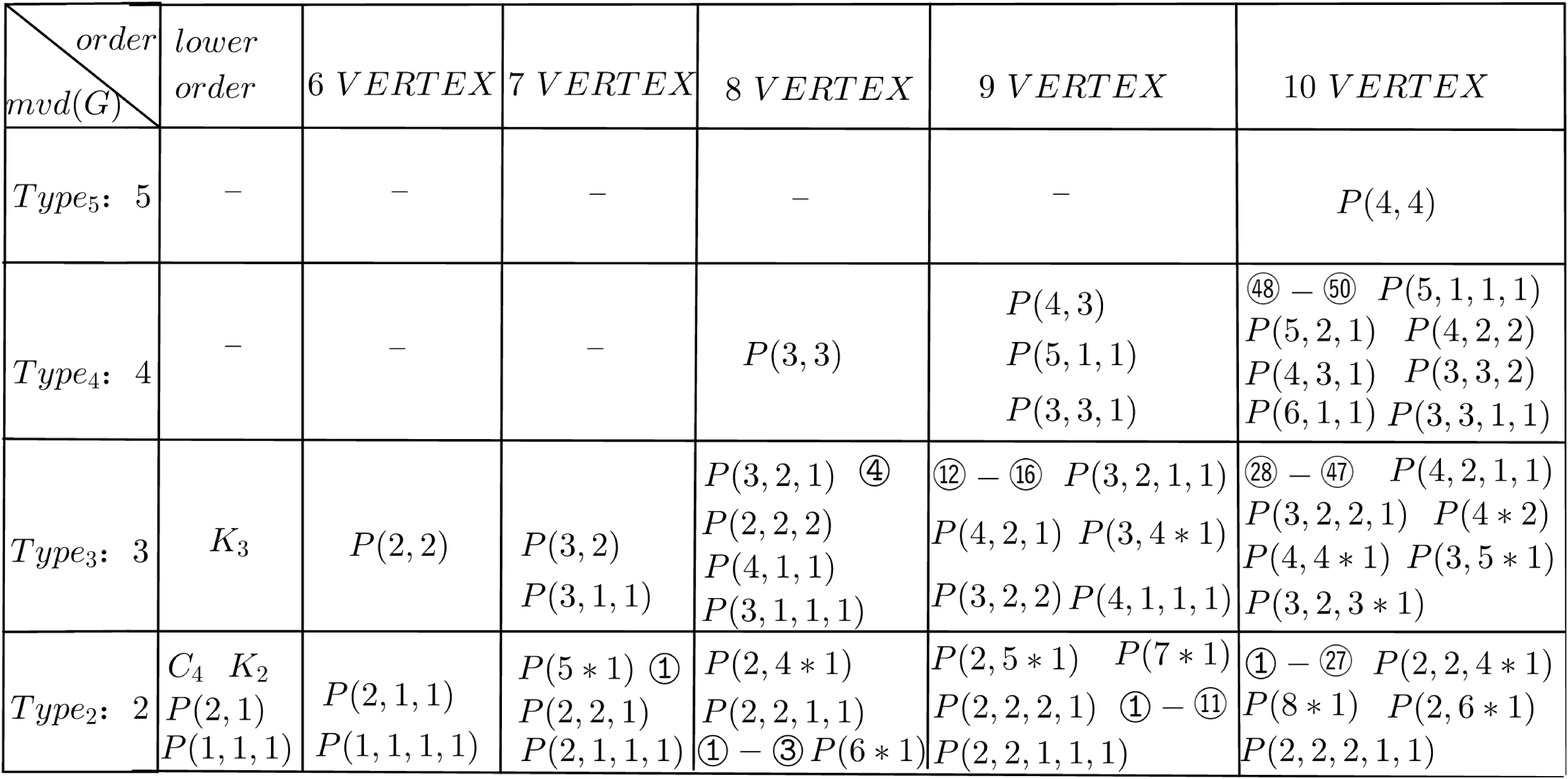}\\
  \label{table1}
\end{figure}

The $nontrivial$ $block$-$induced$ $subgraph$ of a connected graph $G$ is the subgraph induced by all nontrivial blocks of $G$. Let $\mathscr{A}$ (or $\mathscr{B}$, or $\mathscr{C}$) be a set of connected graphs, and the nontrivial block-induced subgraph of each connected graph in $\mathscr{A}$ (or $\mathscr{B}$, or $\mathscr{C}$) is exactly isomorphic to one of the graphs in Fig. \ref{4}$(a)$ \big(or $(b)$, or $(c)$\big). The blue vertex indicates that two blocks may or may not be adjacent to each other.

\begin{figure}
\centering
  \includegraphics[height=10.5cm]{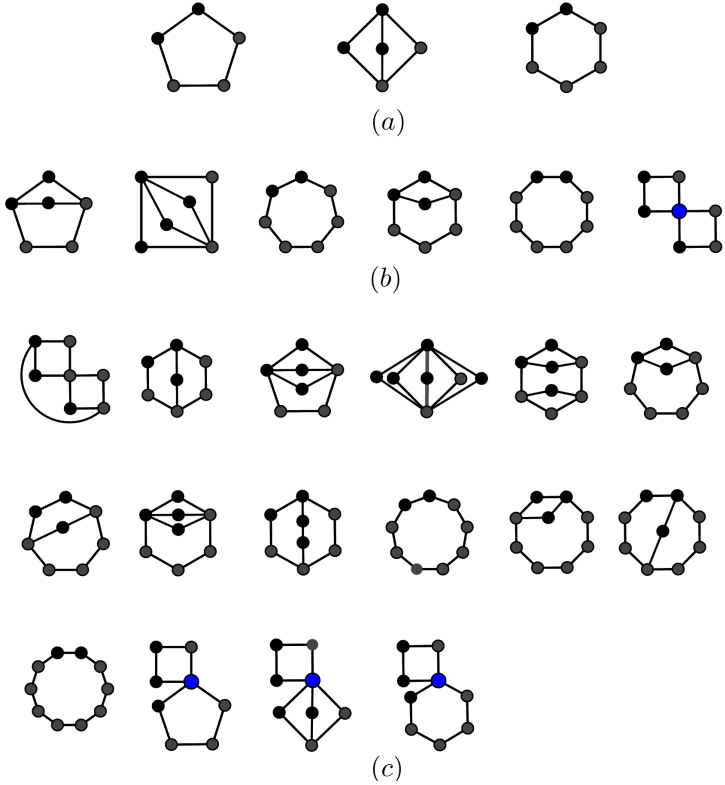}
  \caption{Results for Theorem \ref{large}}
  \label{4}
\end{figure}

\begin{figure}
\centering
  \includegraphics[height=10.5cm]{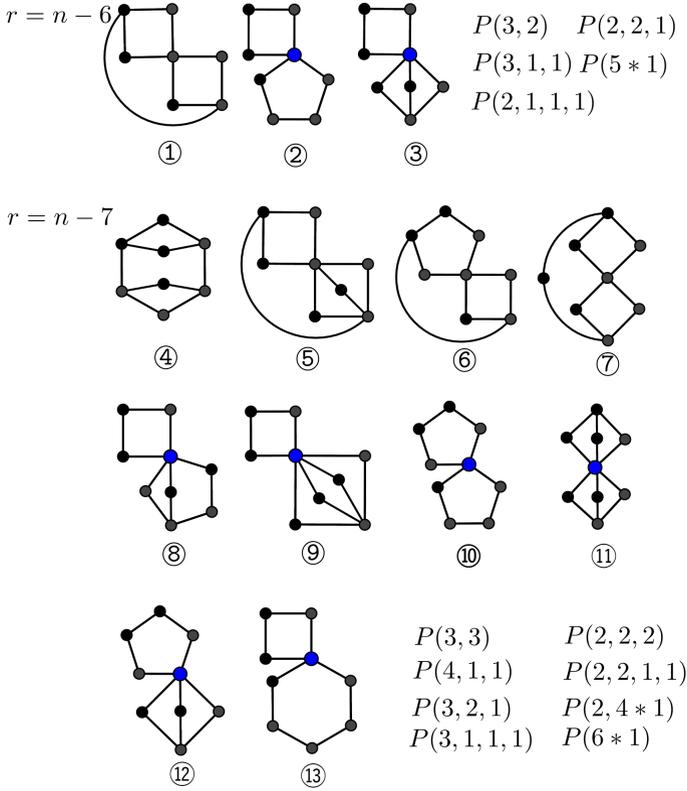}
  \caption{Proof process of Theorem \ref{large}}
  \label{5}
\end{figure}

\begin{proof}[Proof of Theorem \ref{large}]
By Theorem \ref{everyblock}, Corollary \ref{blockrelation} and \ref{6-10}, it is easy to verify sufficiency. Therefore, we only need to prove the necessity. If $mvd(G)=n$, we define a coloring $\tau: V(G)\rightarrow [n]$ such that each vertex is colored differently. Let \{$B_{1}, B_{2}, \ldots , B_{r}$\} be the block decomposition of $G$ and the coloring of $\tau$ restricted to $B_{i}$ be $\tau_{i}$. By Theorem \ref{everyblock}, $\tau_{i}$ is an $mvd$-coloring of $B_{i}$. Then $B_{i}$ is complete. Otherwise, since $B_{i}$ is $2$-connected, at least 2 vertices in $B_{i}$ have the same color, a contradiction. More precisely $G$ is a tree, since $G$ is triangle-free.

Now suppose $mvd(G)\geq n-5$, $G$ has $r$ blocks, and $t$ of them are trivial blocks. By Theorem \ref{blockbound}, we have $n-5\leq mvd(G)\leq \left\lfloor {\frac{n+2t-r+1}{2}}\right\rfloor$. There are two cases.

\textbf{Case 1}. $n-r$ is even.

We have $2t\geq n+r-10$. Since $t\leq r$, $r\geq n-10$. Thus $r$ may be $n-2, n-4, n-6, n-8, n-10$.

$r=n-2$. Since $2t\geq n+r-10$ and $r=t$ if and only if $G$ is a tree (i.e., $r=n-1$), $n-6\leq t\leq n-3$. Since $G$ is triangle-free, we get $\sum\limits_{i=1}^{r} |B_{i}|$ greater than or equal to $4+2(n-3)$, $8+2(n-4)$, $12+2(n-5)$, $16+2(n-6)$ when $t$ is equal to $n-3$, $n-4$, $n-5$, $n-6$, respectively. All the four cases above contradict $n=\big(\sum\limits_{i=1}^{r} |B_{i}|\big)-r+1$. Thus there is no graph that satisfies the condition when $r=n-2$. Similarly, in all other cases of this proof, we can first use this method to make a basic determination of the number of nontrivial blocks in $G$.

$r=n-4$. Similarly, $n-7\leq t\leq n-5$. In fact, there can be only one nontrivial block in $G$. Since $n=\big(\sum\limits_{i=1}^{r} |B_{i}|\big)-r+1$, the order of this block is $5$. The minimal $2$-connected graph of order $5$ is $P(2,1)$ or $P(1,1,1)$. According to Theorem \ref{everyblock}, Corollary \ref{blockrelation} and \ref{6-10}, we have $mvd(G)=n-3$ in both cases.

$r=n-6$. Similarly, $n-8\leq t\leq n-7$. In fact, there are at most two nontrivial blocks in $G$. Since $n=\big(\sum\limits_{i=1}^{r} |B_{i}|\big)-r+1$, if $t=n-7$, the order of this nontrivial block is $7$; if $t=n-8$, there are two nontrivial blocks $B_{i}$ and $B_{j}$, where $|B_{i}|+|B_{j}|=9$. Since $G$ is triangle-free, the nontrivial block-induced subgraph of $G$ is one of the graphs depicted in Fig. \ref{5}. According to Theorem \ref{everyblock}, Corollary \ref{blockrelation} and \ref{6-10}, $mvd(G)=n-4$ when the nontrivial block-induced subgraph of $G$ is $P(3,2)$ or $P(3,1,1)$, and $mvd(G)=n-5$ for the rest cases.

$r=n-8$. Similarly, $t=n-9$. $G$ has exactly one nontrivial block of order $9$ since $n=\big(\sum\limits_{i=1}^{r} |B_{i}|\big)-r+1$. The nontrivial block is one of the graphs depicted in Appendix A. $9$ VERTEX. According to Theorem \ref{everyblock}, Corollary \ref{blockrelation} and \ref{6-10}, $mvd(G)=n-5$ when the nontrivial block-induced subgraph of $G$ is one of $\big \{P(4,3),\ P(5,1,1),\ P(3,3,1)\big \}$, and $mvd(G)<n-5$ for the rest cases.

$r=n-10$. Similarly, $n-10\leq t\leq n-11$, a contradiction.

\textbf{Case 2}. $n-r$ is odd.

We have $2t\geq n+r-11$. Since $t\leq r$, $r\geq n-11$. Thus $r$ may be $n-1, n-3, n-5, n-7, n-9, n-11$. The discussion and calculation process of Case $2$ is similar to that of Case $1$.

$r=n-1$. $G$ is a tree and $mvd(G)=n$.

$r=n-3$. Since $2t\geq n+r-11$ and $r=t$ if and only if $G$ is a tree (i.e., $r=n-1$), $n-7\leq t\leq n-4$. Since $G$ is triangle-free, we get $\sum\limits_{i=1}^{r} |B_{i}|$ greater than or equal to $8+2(n-5)$, $12+2(n-6)$, $16+2(n-7)$ when $t$ is equal to $n-5$, $n-6$, $n-7$, respectively. All the three cases above contradict $n=\big(\sum\limits_{i=1}^{r} |B_{i}|\big)-r+1$. Thus $G$ has exactly one nontrivial block. The order of this block is $4$ since $n=\big(\sum\limits_{i=1}^{r} |B_{i}|\big)-r+1$. Since the minimal $2$-connected graph of order $4$ is $C_{4}$, $G$ is a unicycle graph with cycle $C_{4}$. According to Corollary \ref{blockrelation}, $mvd(G)=n-2$.

$r=n-5$. Similarly, $n-8\leq t\leq n-6$. In fact, there are at most two nontrivial blocks in $G$. Since $n=\big(\sum\limits_{i=1}^{r} |B_{i}|\big)-r+1$ and $G$ is triangle-free, if $t=n-6$, then the order of this nontrivial block is $6$, that is, it can be one of $\{P(2,2),\ P(2,1,1),\ P(1,1,1,1)\}$; if $t=n-7$, then there are two nontrivial blocks $B_{i}$ and $B_{j}$, where $|B_{i}|+|B_{j}|=8$, that is, both nontrivial blocks of $G$ are $C_{4}$ and these two blocks may or may not be adjacent to each other. According to Theorem \ref{everyblock}, Corollary \ref{blockrelation} and \ref{6-10}, $mvd(G)=n-3$ when the nontrivial block-induced subgraph of $G$ is $P(2,2)$, and $mvd(G)=n-4$ for the rest cases.

$r=n-7$. Similarly, $n-9\leq t\leq n-8$. In fact, there are at most two nontrivial blocks in $G$. Since $n=\big(\sum\limits_{i=1}^{r} |B_{i}|\big)-r+1$, if $t=n-8$, the order of this nontrivial block is $8$; if $t=n-9$, there are two nontrivial blocks $B_{i}$ and $B_{j}$, where $|B_{i}|+|B_{j}|=10$. Since $G$ is triangle-free, the nontrivial block-induced subgraph of $G$ is one of the graphs depicted in Fig. \ref{5}. According to Theorem \ref{everyblock}, Corollary \ref{blockrelation} and \ref{6-10}, $mvd(G)=n-4$ when the nontrivial block-induced subgraph of $G$ is $P(3,3)$, $mvd(G)=n-5$ when the nontrivial block-induced subgraph of $G$ is one of $\big \{\Large{\textcircled{\small {4}}},\ \Large{\textcircled{\small {13}}},\ P(4,1,1),\ P(3,2,1),\ P(3,1,1,1),\ P(2,2,2)\big \}$, and $mvd(G)=n-6$ for the rest cases.

$r=n-9$. Similarly, $t=n-10$. $G$ has exactly one nontrivial block of order $10$ since $n=\big(\sum\limits_{i=1}^{r} |B_{i}|\big)-r+1$. The nontrivial block is one of the graphs depicted in Appendix A. $10$ VERTEX. According to Theorem \ref{everyblock}, Corollary \ref{blockrelation} and \ref{6-10}, $mvd(G)=n-5$ when the nontrivial block-induced subgraph of $G$ is $P(4,4)$, and $mvd(G)<n-5$ for the rest cases.

$r=n-11$. Similarly, $n-11\leq t\leq n-12$, a contradiction.
\end{proof}

\begin{rem}\label{remark 4}
If all (nontrivial) blocks of a connected graph $G$ are triangles, then $mvd(G)=n$.
\end{rem}

\subsection{ An example }
\quad \quad We conclude this paper with an example, which is also an application of Corollary \ref{6-10} and Algorithm $3$.

\noindent$\displaystyle \textbf{Example 3.1.}$
If $G$ is any graph whose blocks are all minimally $2$-connected graphs of small order, combining Algorithm $2$ and Algorithm $4$, we can compute $mvd(G)$ and give an $mvd$-coloring of $G$ in polynomial time.

\makeatletter
\newenvironment{breakable3algorithm}
  {
   \begin{center}
     \refstepcounter{algorithm}
     \hrule height.8pt depth0pt \kern2pt
     \renewcommand{\caption}[2][\relax]{
       {\raggedright\textbf{\ALG@name~\thealgorithm} ##2\par}
       \ifx\relax##1\relax
         \addcontentsline{loa}{algorithm}{\protect\numberline{\thealgorithm}##2}
       \else
         \addcontentsline{loa}{algorithm}{\protect\numberline{\thealgorithm}##1}
       \fi
       \kern2pt\hrule\kern2pt
     }
  }{
     \kern2pt\hrule\relax
   \end{center}
  }
\makeatother

\begin{breakable3algorithm}
\caption{An example}
\hspace*{0.02in} {\bf Input:}
$CutVerticesSet$ and $BlocksSet$ from Algorithm $2$, Appendix $A$\\
\hspace*{0.02in} {\bf Output:}
$mvd(G)$, $mvd$-coloring of $G$
\begin{algorithmic}[1]
\State $n_5 \gets 0,\ n_4 \gets 0,\ n_3 \gets 0,\ n_2 \gets 0$
\State $mvd(G) \gets 0$
\For{\textbf{every} block $Block \in BlocksSet$}
    \State find the block $thBlock$ isomorphic to $Block$ in Appendix $A$
    \For{$Type_i$ in Corollary \ref{6-10}}
        \If{$thBlock \in Type_i$}
            \State $n_i++$
        \EndIf
    \EndFor
    \State color $Block$ with new colors as $thBlock$ in Appendix $A$
\EndFor
$mvd(G) \gets 4n_5+3n_4+2n_3+n_2+1$
\For{\textbf{every} vertex $u \in CutVerticesSet$}
    \State color($u$) $\gets$ null
\EndFor
\For{\textbf{every} block $Block \in BlocksSet$}
    \For{\textbf{every} vertex $v \in Block$}
        \If{$v \in CutVerticesSet$}
           \State $CutVertex \gets $cut-vertex in $CutVerticesSet$ corresponding to $v$
           \If{color($CutVertex$)=null}
               \State color($CutVertex$) $\gets$ color($v$)
           \Else
               \State color($v$) $\gets$ color($CutVertex$)
               \State update the colors of all vertices in $Block$ according to $thBlock$ in Appendix $A$
           \EndIf
        \EndIf
    \EndFor
\EndFor
\end{algorithmic}
\end{breakable3algorithm}

\begin{thm}\label{np}
we can compute $mvd(G)=4n_5+3n_4+2n_3+n_2+1$ and give an $mvd$-coloring of $G$ in polynomial time.
\end{thm}
\begin{proof}
By Theorem \ref{everyblock}, Corollary \ref{blockrelation} and \ref{6-10}, $mvd(G)=4n_5+3n_4+2n_3+n_2+1$ and the coloring given by Algorithm $4$ is an $mvd$-coloring of $G$. Note the time complexity of Algorithm $2$ is $O(e)$, where $e=|E(G)|$. This is because that each edge is still scanned exactly once in each direction. The number of operations per edge is bounded by a constant, except when a block is produced. Each vertex is pushed into $T$ once, and popped once. Thus, the total time to produce the blocks is $O(n)$, where $n=|V(G)|$. When the order of $Block$ is small, the number of operations on line $4$ is bounded by a constant. Therefore, the complexity of Algorithm $4$ is still $O(n)$ even if we need to determine whether two graphs are isomorphic in line $4$, and we know that it is a $NP$-problem to determine whether two graphs are isomorphic.
\end{proof}

\section*{Acknowledgement}

\quad \quad This work is supported by the National Natural Science Foundation of China (NSFC11921001), the Natural Key Research and Development Program of China (2018YFA0704701), the National Natural Science Foundation of China (NSFC11801410) and the National Natural Science Foundation of China (NSFC11971346). The authors are grateful to professor C. Zong for his supervision and discussion.

\newpage
\begin{appendix}
\section{Minimal blocks With 10 and Fewer Vertices}

\begin{figure}[H]
\centering
  \includegraphics[width=11.5cm]{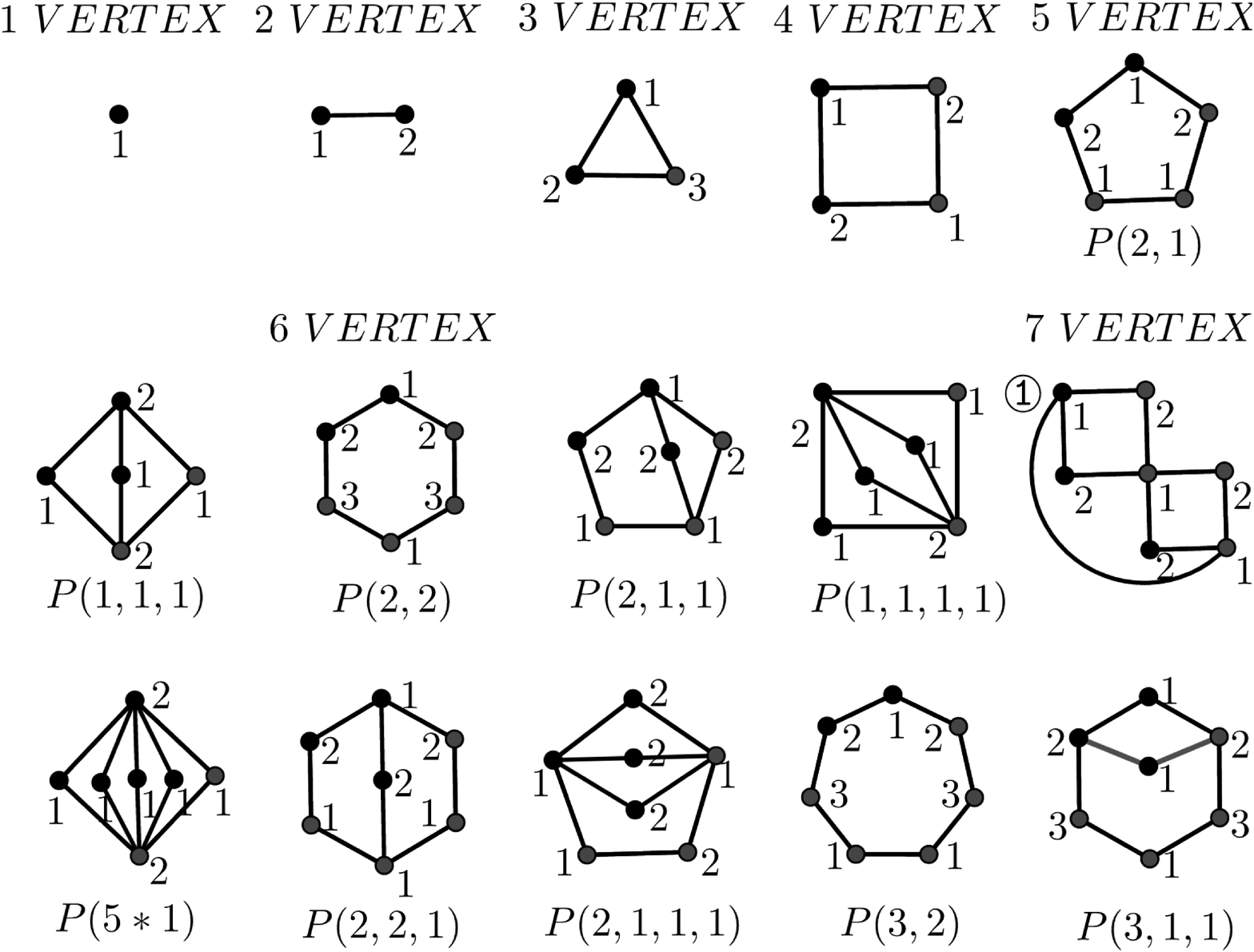}
  \includegraphics[width=12cm]{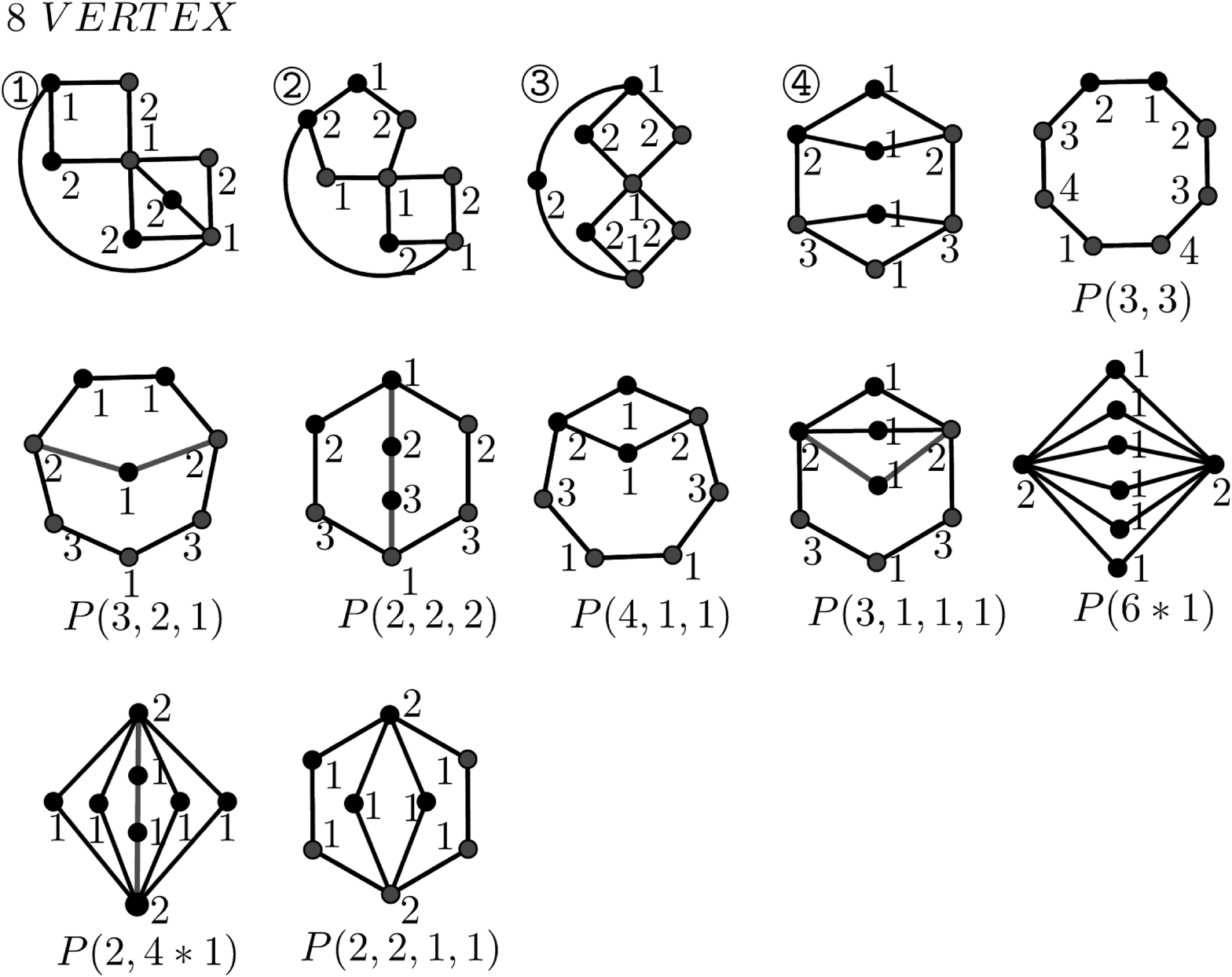}
\end{figure}
\begin{figure}[H]
\centering
  \includegraphics[height=8cm]{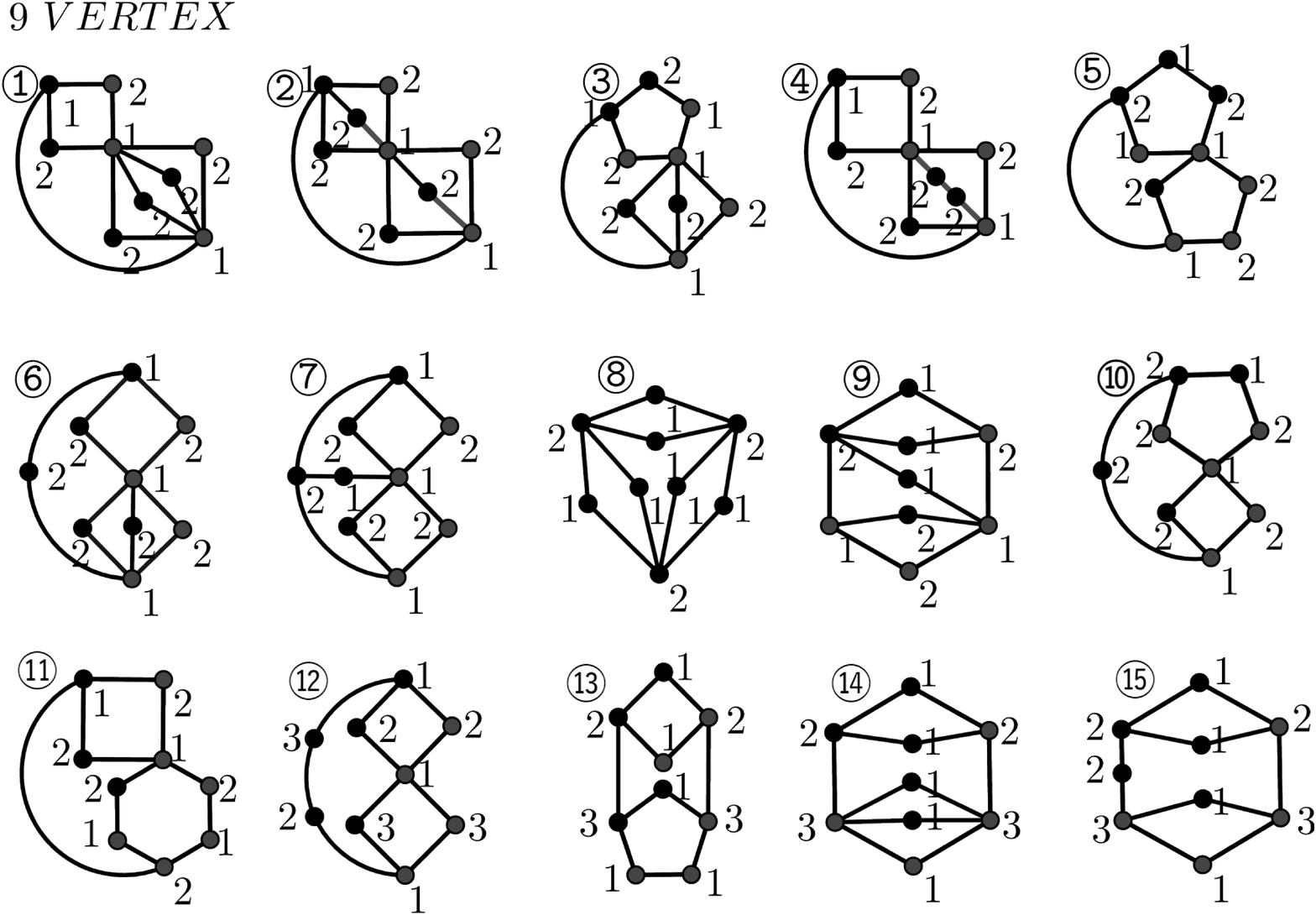}
  \includegraphics[height=10.5cm]{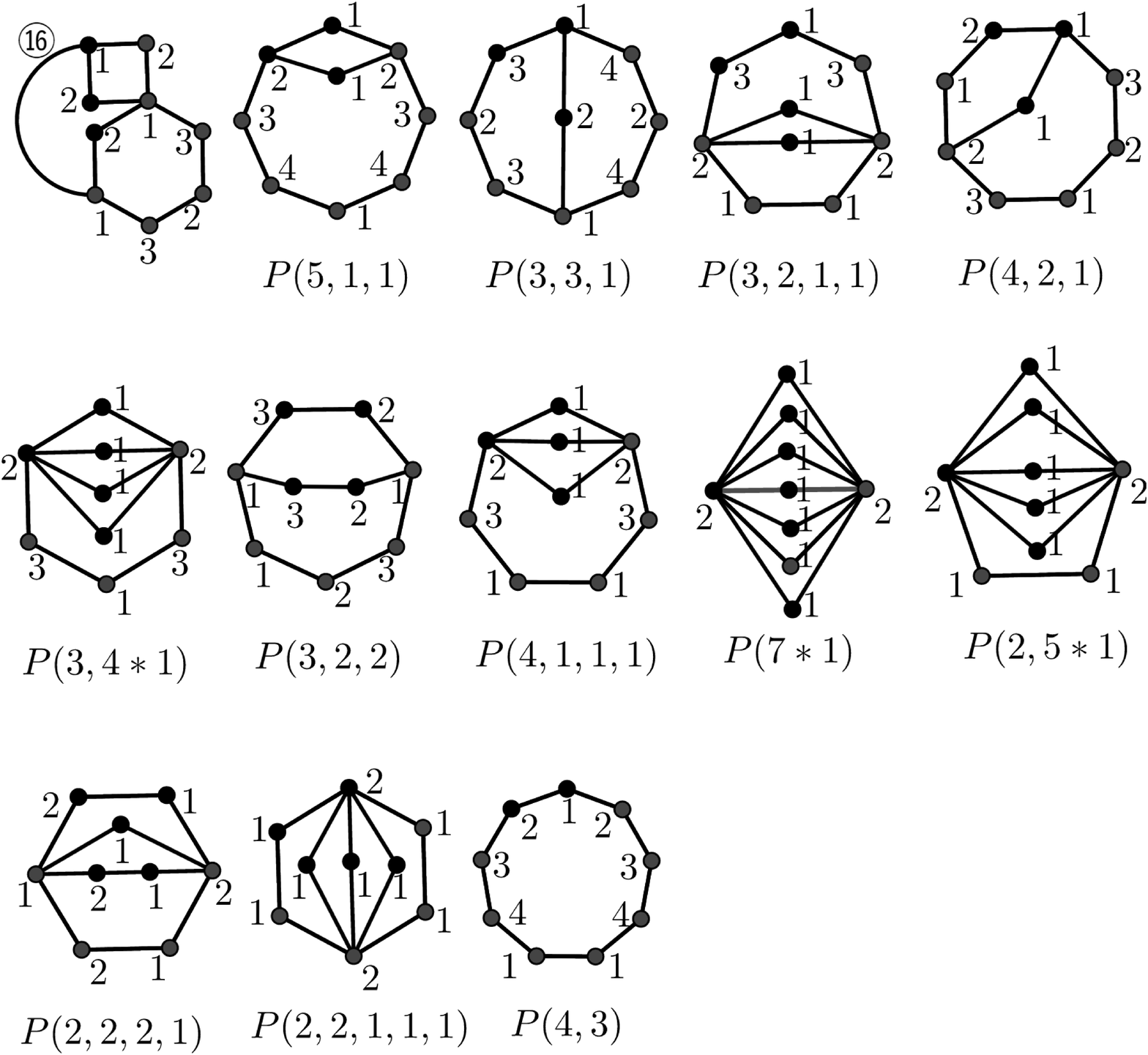}
\end{figure}
\begin{figure}[H]
\centering
  \includegraphics[height=10cm]{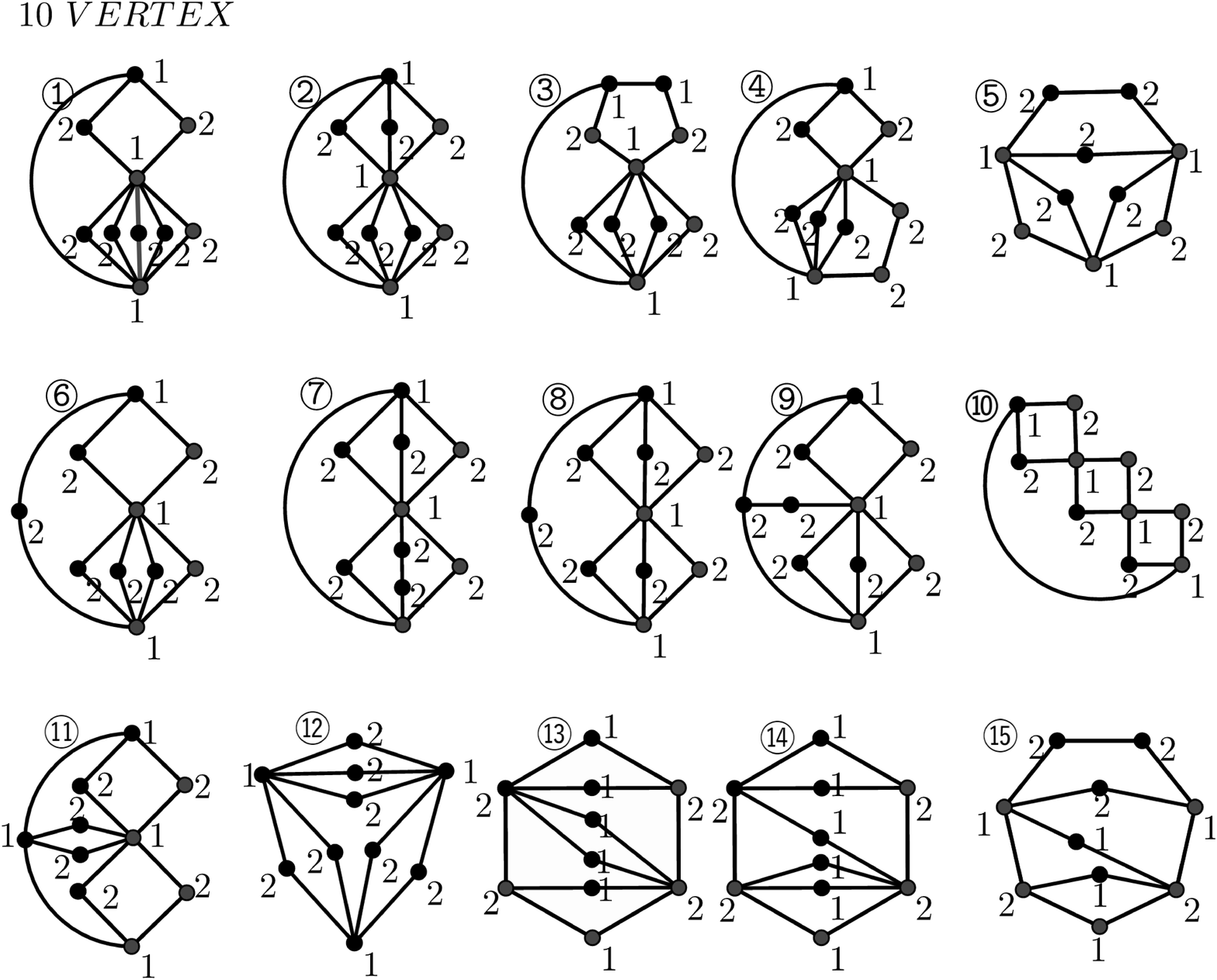}
  \includegraphics[height=7cm]{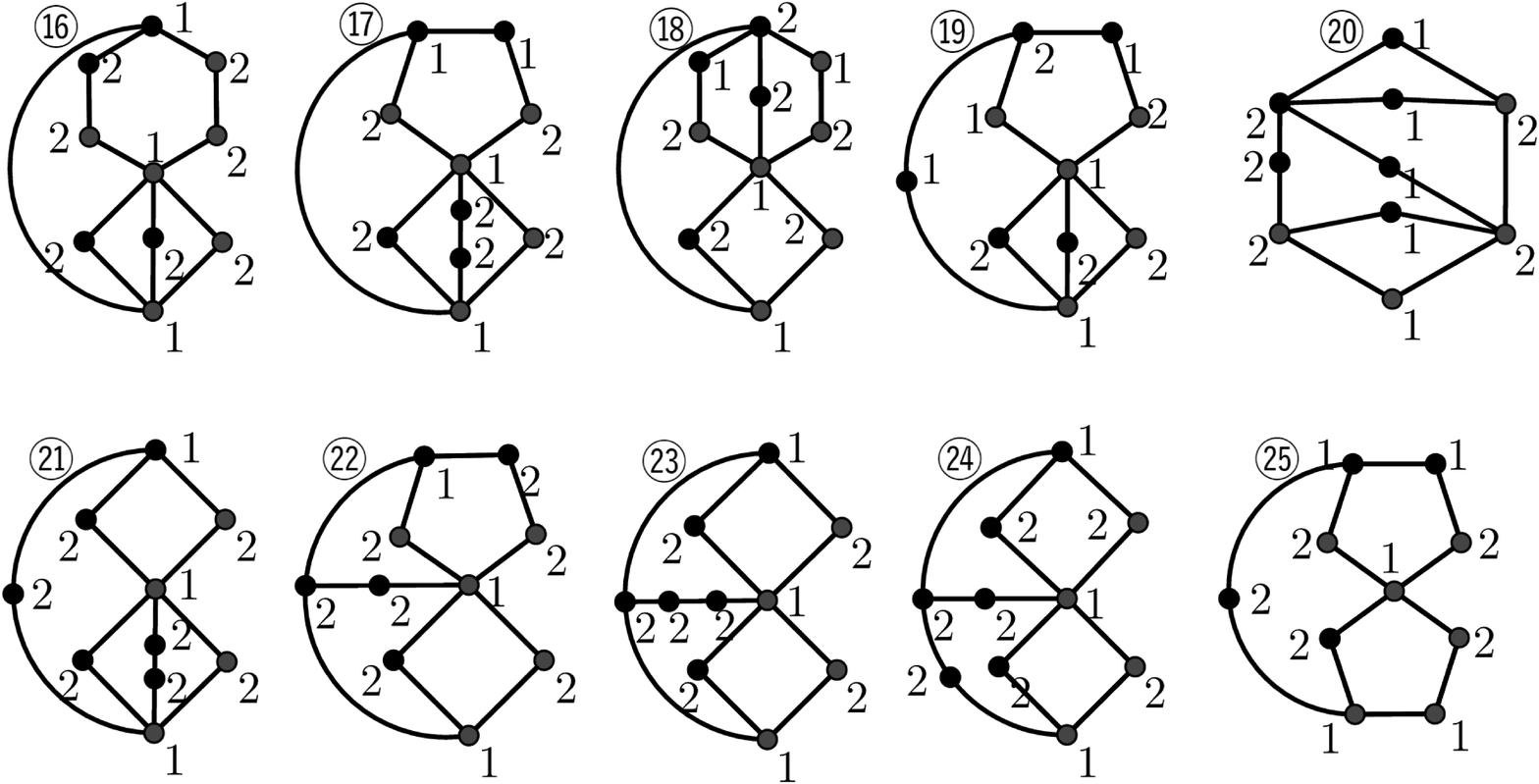}
  \includegraphics[height=3cm]{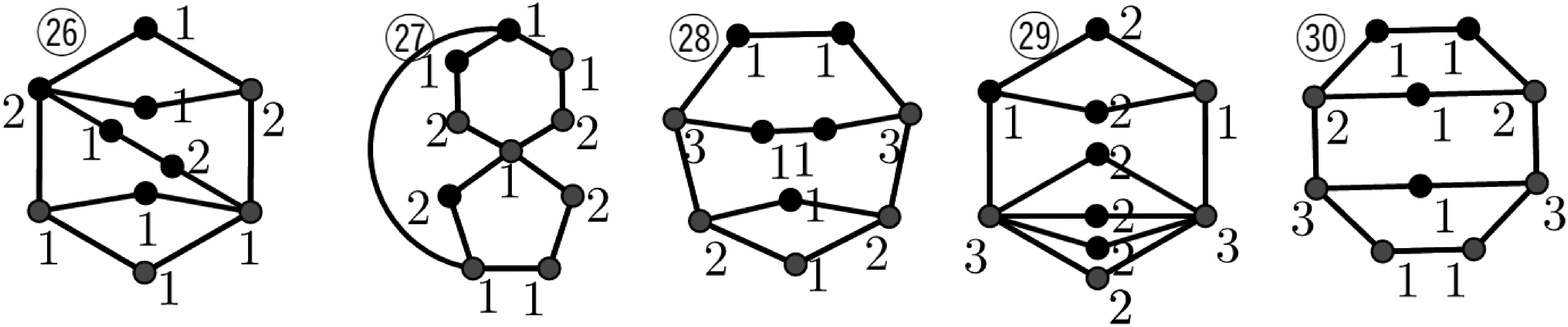}
\end{figure}
\begin{figure}[H]
\centering
  \includegraphics[height=6.2cm]{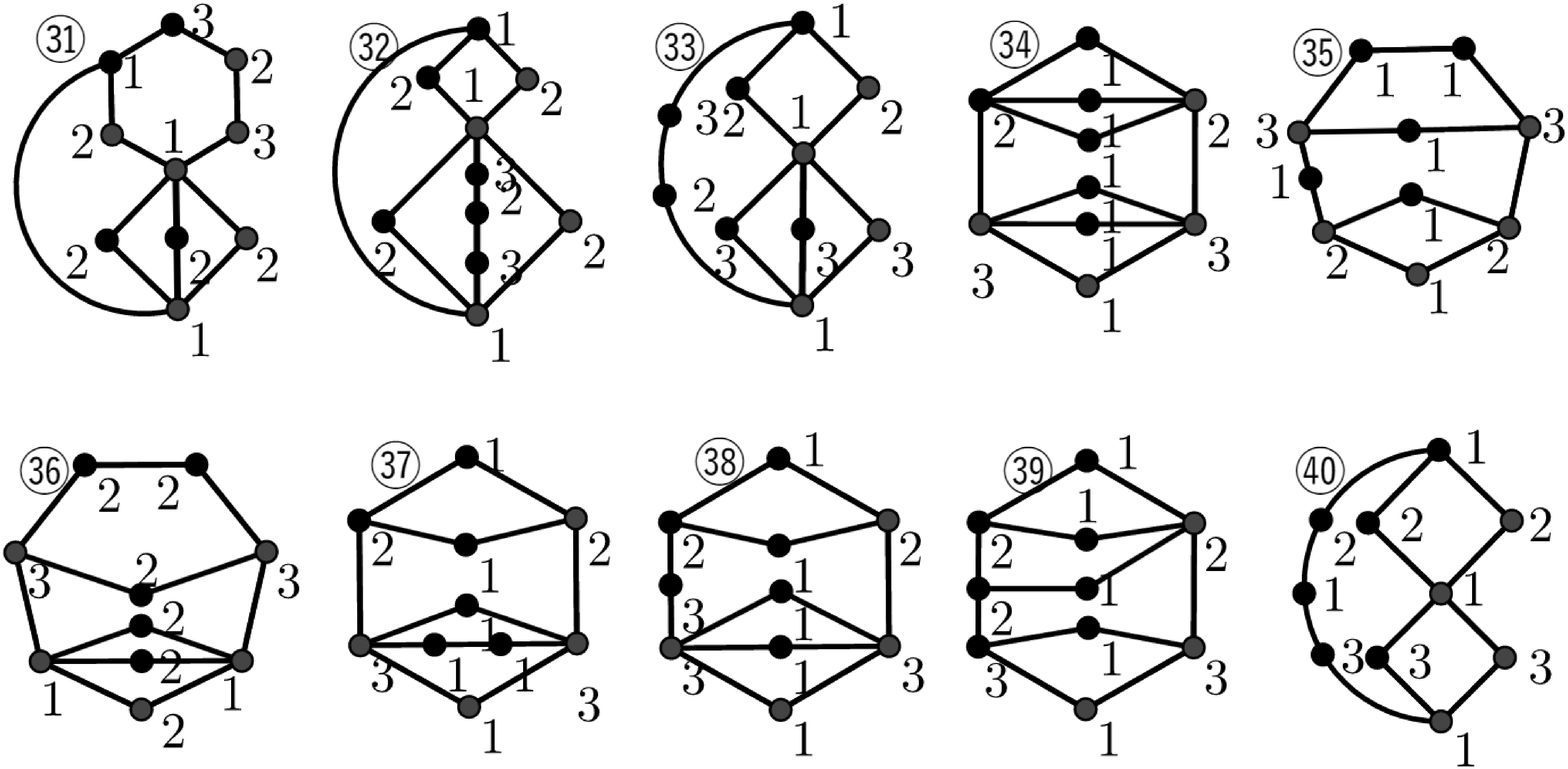}
  \includegraphics[height=7.2cm]{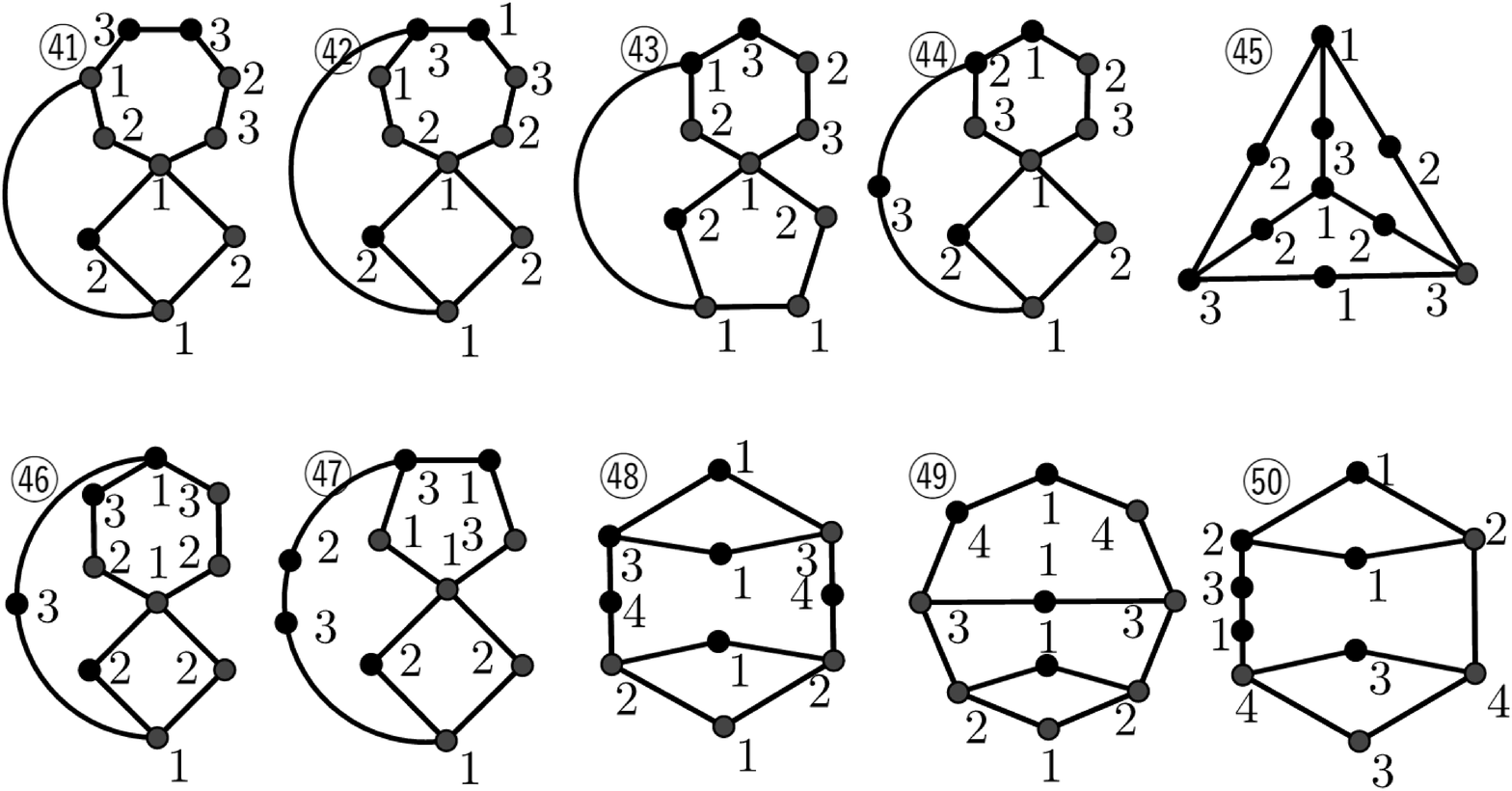}
  \includegraphics[height=7.9cm]{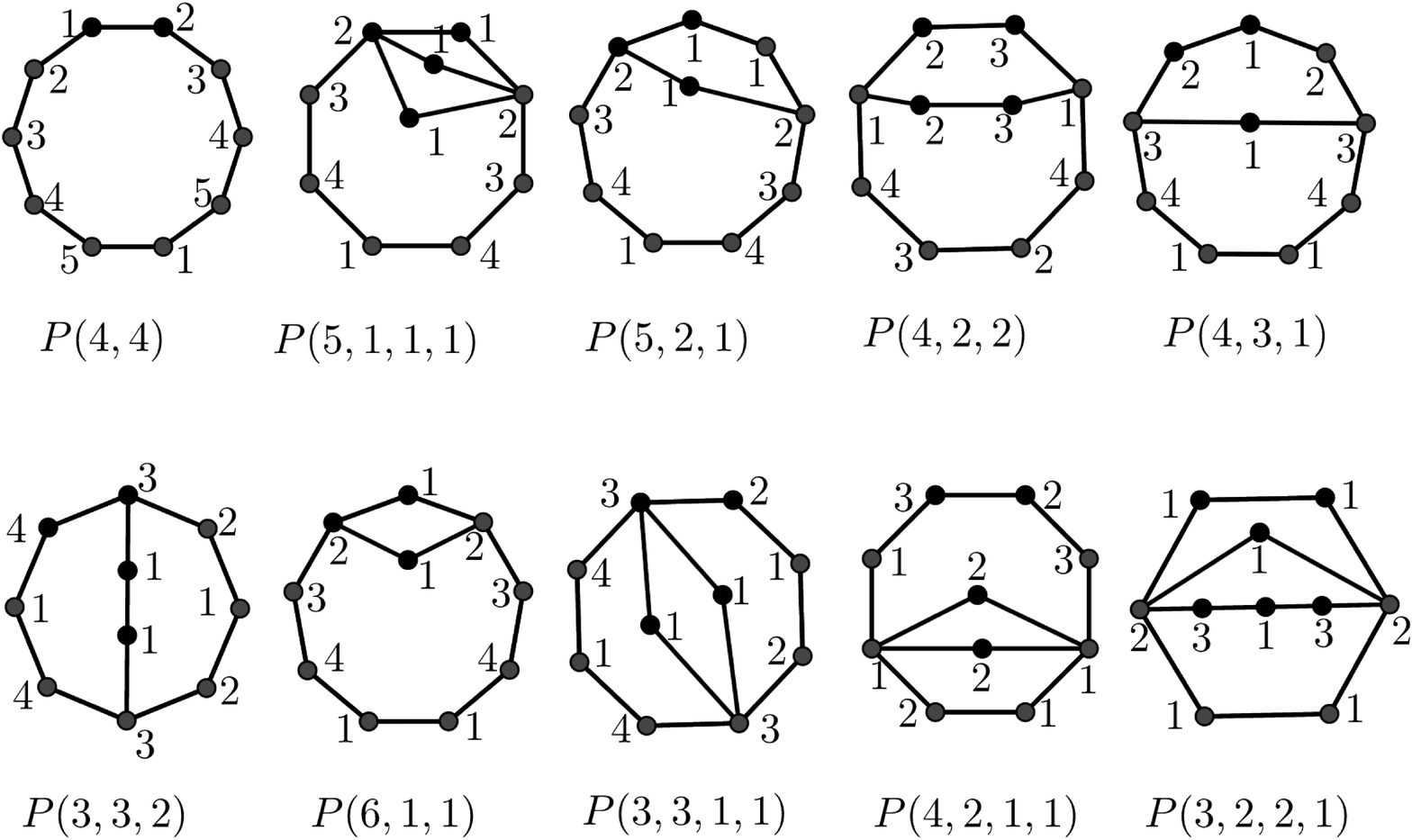}
\end{figure}
\begin{figure}[H]
\centering
  \includegraphics[height=8.4cm]{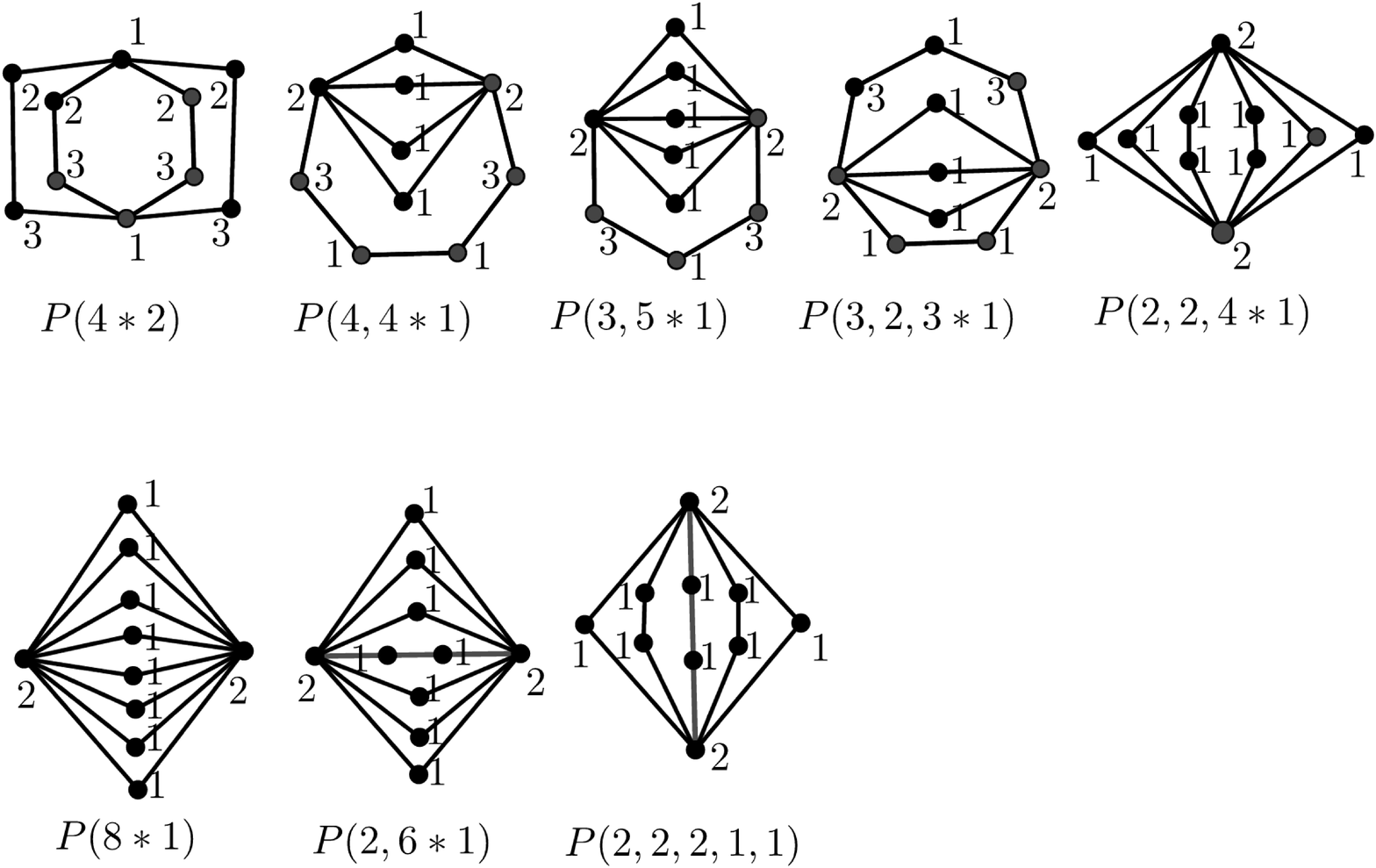}
\end{figure}

\newpage
\section{Code for $mvd$-coloring}

\begin{lstlisting}[ language=Java]
public class GraphContext {
    public static void main(String[] args) {
        String[] verx = new String[]{"A", "B", "C", "D", "E", "F", "G", "H", "I", "J", "K", "L", "M", "N", "O", "P", "Q"};
        int[][] edges = new int[][]{
                {0, 0, 0, 0, 1, 0, 0, 0, 0, 0, 0, 0, 0, 0, 0, 1, 0},
                {0, 0, 0, 0, 0, 0, 0, 1, 0, 0, 0, 0, 0, 0, 0, 0, 1},
                {0, 0, 0, 0, 0, 0, 0, 0, 0, 0, 0, 1, 0, 0, 1, 0, 0},
                {0, 0, 0, 0, 0, 0, 0, 0, 0, 0, 0, 1, 1, 0, 0, 0, 0},
                {1, 0, 0, 0, 0, 1, 0, 0, 0, 0, 1, 0, 0, 1, 0, 0, 0},
                {0, 0, 0, 0, 1, 0, 1, 0, 0, 0, 0, 0, 0, 0, 0, 0, 0},
                {0, 0, 0, 0, 0, 1, 0, 1, 0, 1, 0, 0, 0, 0, 0, 1, 0},
                {0, 1, 0, 0, 0, 0, 1, 0, 0, 0, 0, 0, 1, 1, 1, 0, 0},
                {0, 0, 0, 0, 0, 0, 0, 0, 0, 0, 0, 1, 1, 0, 0, 0, 0},
                {0, 0, 0, 0, 0, 0, 1, 0, 0, 0, 0, 0, 0, 1, 0, 0, 0},
                {0, 0, 0, 0, 1, 0, 0, 0, 0, 0, 0, 0, 0, 0, 0, 1, 0},
                {0, 0, 1, 1, 0, 0, 0, 0, 1, 0, 0, 0, 0, 0, 0, 0, 1},
                {0, 0, 0, 1, 0, 0, 0, 1, 1, 0, 0, 0, 0, 0, 0, 0, 0},
                {0, 0, 0, 0, 1, 0, 0, 1, 0, 1, 0, 0, 0, 0, 0, 0, 0},
                {0, 0, 1, 0, 0, 0, 0, 1, 0, 0, 0, 0, 0, 0, 0, 0, 0},
                {1, 0, 0, 0, 0, 0, 1, 0, 0, 0, 1, 0, 0, 0, 0, 0, 0},
                {0, 1, 0, 0, 0, 0, 0, 0, 0, 0, 0, 1, 0, 0, 0, 0, 0}
        };
        BlockAndCutVerticesBuilder blockAndCutVerticesBuilder = new BlockAndCutVerticesBuilder();
        Edge[][] matrixEdges = blockAndCutVerticesBuilder.create(verx, edges);
        System.out.println("### Input Graph ###");
        blockAndCutVerticesBuilder.printGraph(matrixEdges);
        //use tarjan algorithm to calculate cutVertices and blocks
        blockAndCutVerticesBuilder.makeDFSTarjan(0);
        //print cutVertices, blocks and Adjacency matrix
        blockAndCutVerticesBuilder.printResult();
        System.out.println();
        List<MvdGraph> graphs = blockAndCutVerticesBuilder.createMvdGraph();
        //mvd coloring
        MvdColorMarker mvdColorMarker = new MvdColorMarker();
        mvdColorMarker.markBlocks(graphs);
        Node[] vertices = blockAndCutVerticesBuilder.vertices;
        System.out.println("### Coloring Vertices Results ###");
        for (Node node : vertices) {
            System.out.print(node);
        }
    }
}
------------------------------------------------
public class BlockAndCutVerticesBuilder {
    public Node[] vertices;
    public Edge[][] edges;
    public Set<Node> cutVerticesSet = new HashSet<>();
    public List<List<Node>> block = new ArrayList<>();
    public List<Edge[][]> blocksEdges = new ArrayList<>();
    public Stack<Node> nodesStack = new Stack<>();
    int root;
    int count = 1;
    public Edge[][] create(String[] vertexs, int[][] edges) {
        this.vertices = new Node[vertexs.length];
        for (int i = 0; i < vertexs.length; i++) {
            this.vertices[i] = new Node(vertexs[i]);
        }
        this.edges = new Edge[edges.length][edges[0].length];
        for (int row = 0; row < edges.length; row++) {
            for (int col = 0; col < edges[0].length; col++) {
                this.edges[row][col] = new Edge(edges[row][col]);
            }
        }
        return this.edges;
    }
    public List<MvdGraph> createMvdGraph() {
        List<MvdGraph> graphs = new ArrayList<>();
        for (int i = 0; i < block.size(); i++) {
            graphs.add(new MvdGraph(block.get(i).toArray(new Node[0]), blocksEdges.get(i)));
        }
        return graphs;
    }
    public void printGraph(Edge[][] edges) {
        for (int row = 0; row < edges.length; row++) {
            for (int col = 0; col < edges[0].length; col++) {
                System.out.printf("%d\t", edges[row][col].connected);
            }
            System.out.println();
        }
    }
    public void printResult() {
        System.out.println("### CutVertices and Blocks  ###");
        System.out.println("cutVertices of Input Graph:");
        System.out.println("cutVerticesSet:" + cutVerticesSet);
        System.out.println();
        System.out.println("Block generated from Graph:");
        for (int i = 0; i < block.size(); i++) {
            System.out.println("Block num " + (i + 1));
            System.out.println(block.get(i));
            PrintUtils.printGraph(blocksEdges.get(i));
        }
    }
    public void makeDFSTarjan(int activeNodeIndex) {
        vertices[activeNodeIndex].depth = 1;
        nodesStack.push(vertices[activeNodeIndex]);
        root = activeNodeIndex;
        //calculate the cut vertices and blocks
        DFSTarjan(activeNodeIndex);
        //calculate the blocks' adjacency matrices
        cutBlocksEdges();
    }
    public void cutBlocksEdges() {
        for (List<Node> lists : block) {
            List<Integer> indices = new ArrayList<Integer>();
            for (Node node : lists) {
                indices.add(getIndexOfNode(node));
            }
            Edge[][] blockEdge = new Edge[indices.size()][indices.size()];
            for (int i = 0; i < indices.size(); i++) {
                for (int j = 0; j < indices.size(); j++) {
                    int rowIndex = indices.get(i);
                    int colIndex = indices.get(j);
                    blockEdge[i][j] = edges[rowIndex][colIndex];
                }
            }
            blocksEdges.add(blockEdge);
        }
    }
    private void DFSTarjan(int activeNodeIndex) {
        Node currentNode = this.vertices[activeNodeIndex];
        while(this.getNodesOfUnreachedEdge(activeNodeIndex) != null || currentNode.parent != null) {
            if (this.getNodesOfUnreachedEdge(activeNodeIndex) == null) {
                if (currentNode.low >= currentNode.parent.depth) {
                    if (this.getIndexOfNode(currentNode.parent) != this.root || this.getNodesOfUnreachedEdge(this.root) != null) {
                        this.cutVerticesSet.add(currentNode.parent);
                        currentNode.parent.isCutVertex = true;
                    }
                    ArrayList list = new ArrayList();
                    while(!currentNode.equals(this.nodesStack.peek())) {
                        list.add(this.nodesStack.pop());
                    }
                    list.add(this.nodesStack.pop());
                    list.add(currentNode.parent);
                    this.block.add(list);
                } else {
                    currentNode.parent.low = Math.min(currentNode.parent.low, currentNode.low);
                }
                return;
            }
            Node nextNode = this.getNodesOfUnreachedEdge(activeNodeIndex);
            this.markEdgeReached(currentNode, nextNode);
            if (nextNode.depth == 0) {
                this.nodesStack.push(nextNode);
                nextNode.parent = currentNode;
                ++this.count;
                nextNode.depth = this.count;
                nextNode.low = this.count;
                this.DFSTarjan(this.getIndexOfNode(nextNode));
            } else {
                currentNode.low = Math.min(currentNode.low, nextNode.depth);
            }
        }
    }
    private void markEdgeReached(Node currentNode, Node nextNode) {
        int from = this.getIndexOfNode(currentNode);
        int to = this.getIndexOfNode(nextNode);
        this.edges[from][to].reached = true;
        this.edges[to][from].reached = true;
    }
    private Node getNodesOfUnreachedEdge(int activeVex) {
        for(int col = 0; col < this.edges[activeVex].length; ++col) {
            Edge edge = this.edges[activeVex][col];
            if (edge.connected == 1) {
                Node cur = this.vertices[col];
                if (!edge.reached) {
                    return cur;
                }
            }
        }
        return null;
    }
    private int getIndexOfNode(Node node) {
        for(int i = 0; i < this.vertices.length; ++i) {
            if (this.vertices[i].name.equals(node.name)) {
                return i;
            }
        }
        return -1;
    }
}
------------------------------------------------
public class IsomorphicJudger {
    public static Node[] getIsomorphicColors(MvdGraph block, MvdGraph type) {
        if (block == null || type == null) {
            return null;
        }
        if (block.vertices == null || type.vertices == null) {
            return null;
        }
        if (block.vertices.length != type.vertices.length) {
            return null;
        }
        int n = block.vertices.length;
        //block adjacency matrix
        int[][] blockArray = new int[n][n];
        for (int i1 = 0; i1 < n; i1++) {
            blockArray[i1] = new int[n];
        }
        // block 01 adjacency matrix
        int[][] blockTongxin = new int[n][n];
        for (int i1 = 0; i1 < n; i1++) {
            blockTongxin[i1] = new int[n];
        }
        //block XOR matrix
        int[][] blockYihuo = new int[n][n];
        for (int i1 = 0; i1 < n; i1++) {
            blockYihuo[i1] = new int[n];
        }
        //block XNOR matrix
        int[][] blockTonghuo = new int[n][n];
        for (int i1 = 0; i1 < n; i1++) {
            blockTonghuo[i1] = new int[n];
        }
        //type adjacency matrix
        int[][] typeArray = new int[n][n];
        for (int i2 = 0; i2 < n; i2++) {
            typeArray[i2] = new int[n];
        }
        // type 01 adjacency matrix
        int[][] typeTongxin = new int[n][n];
        for (int i1 = 0; i1 < n; i1++) {
            typeTongxin[i1] = new int[n];
        }
        //block XOR matrix
        int[][] typeYihuo = new int[n][n];
        for (int i1 = 0; i1 < n; i1++) {
            typeYihuo[i1] = new int[n];
        }
        //block XNOR matrix
        int[][] typeTonghuo = new int[n][n];
        for (int i1 = 0; i1 < n; i1++) {
            typeTonghuo[i1] = new int[n];
        }
        for (int row = 0; row < block.edges.length; row++) {
            for (int col = 0; col < block.edges[0].length; col++) {
                blockArray[row][col] = block.edges[row][col].connected;
            }
        }
        for (int row = 0; row < type.edges.length; row++) {
            for (int col = 0; col < type.edges[0].length; col++) {
                typeArray[row][col] = type.edges[row][col].connected;
            }
        }
        oneZero(blockArray, blockTongxin, n);
        xor(blockArray, blockTongxin, blockYihuo, n);
        xnor(blockArray, blockTongxin, blockTonghuo, n);
        oneZero(typeArray, typeTongxin, n);
        xor(typeArray, typeTongxin, typeYihuo, n);
        xnor(typeArray, typeTongxin, typeTonghuo, n);
        return getTransVertices(blockArray, blockYihuo, blockTonghuo, typeArray, typeYihuo, typeTonghuo, n, type.vertices, 0, 0);
    }
    public static Node[] getTransVertices(int[][] blockArray, int[][] blockYihuo, int[][] blockTonghuo,
                                          int[][] typeArray, int[][] typeYihuo, int[][] typeTonghuo, int n,
                                          Node[] typeVertices, int matchedRowNum, int blockRowBegin) {
        int[][] typeArrayCopy = new int[n][n];
        for (int i = 0; i < n; i++) {
            for (int j = 0; j < n; j++) {
                typeArrayCopy[i][j] = typeArray[i][j];
            }
        }
        int[][] typeYihuoCopy = new int[n][n];
        for (int i = 0; i < n; i++) {
            for (int j = 0; j < n; j++) {
                typeYihuoCopy[i][j] = typeYihuo[i][j];
            }
        }
        int[][] typeTonghuoCopy = new int[n][n];
        for (int i = 0; i < n; i++) {
            for (int j = 0; j < n; j++) {
                typeTonghuoCopy[i][j] = typeTonghuo[i][j];
            }
        }
        //colored vertices template
        Node[] vertices = new Node[n];
        for (int i = 0; i < vertices.length; i++) {
            vertices[i] = new Node(typeVertices[i].name, typeVertices[i].color);
        }
        //template array to store data
        int[] tempBlockYihuo = new int[n];
        //template array to store data
        int[] tempTypeYihuo = new int[n];
        //template array to store data
        int[] tempBlockArray = new int[n];
        //template array to store data
        int[] tempBlockTonghuo = new int[n];
        //template array to store data
        int[] tempTypeArray = new int[n];
        //template array to store data
        int[] tempTypeTonghuo = new int[n];
        int t;
        int colNum;
        for (int blockRow = blockRowBegin; blockRow < n; blockRow++) {
            for (int j = 0; j < n; j++) {
                tempBlockYihuo[j] = blockYihuo[blockRow][j];
            }
            for (int j = 0; j < n; j++) {
                tempBlockArray[j] = blockArray[blockRow][j];
            }
            for (int j = 0; j < n; j++) {
                tempBlockTonghuo[j] = blockTonghuo[blockRow][j];
            }
            //bubble sort preprocessing
            bubbleSort(tempBlockYihuo, n);
            bubbleSort(tempBlockArray, n);
            bubbleSort(tempBlockTonghuo, n);
            for (int typeRow = blockRow; typeRow < n; typeRow++) {
                colNum = 0;
                for (int j = 0; j < n; j++) {
                    tempTypeYihuo[j] = typeYihuo[typeRow][j];
                }
                for (int j = 0; j < n; j++) {
                    tempTypeArray[j] = typeArray[typeRow][j];
                }
                for (int j = 0; j < n; j++) {
                    tempTypeTonghuo[j] = typeTonghuo[typeRow][j];
                }
                bubbleSort(tempTypeYihuo, n);
                bubbleSort(tempTypeArray, n);
                bubbleSort(tempTypeTonghuo, n);
                //begin to compare
                for (int col = 0; col < n; col++) {
                    //compare the XOR matrix
                    if (tempBlockYihuo[col] != tempTypeYihuo[col]) {
                        if (typeRow == n - 1) {
                            return null;
                        }
                        break;
                    }
                    colNum = col;
                    if (col != n - 1) {
                        continue;
                    }
                    //compare the adjacency matrix
                    for (int b = 0; b < n; b++) {
                        if (tempBlockArray[b] == tempTypeArray[b]) {
                            continue;
                        } else if (b < n - 1) {
                            return null;
                        }
                    }
                    //compare the XNOR matrix
                    for (int c = 0; c < n; c++) {
                        if (tempBlockTonghuo[c] == tempTypeTonghuo[c]) {
                            continue;
                        } else if (c < n - 1) {
                            return null;
                        }
                    }
                    matchedRowNum++;//two rows match
                    Node tempNode = vertices[blockRow];
                    vertices[blockRow] = vertices[typeRow];
                    vertices[typeRow] = tempNode;
                    //elementary operations: switch adjacency cols
                    for (int c = 0; c < n; c++) {
                        t = typeArrayCopy[blockRow][c];
                        typeArrayCopy[blockRow][c] = typeArrayCopy[typeRow][c];
                        typeArrayCopy[typeRow][c] = t;
                    }
                    //elementary operations: switch adjacency rows
                    for (int r = 0; r < n; r++) {
                        t = typeArrayCopy[r][blockRow];
                        typeArrayCopy[r][blockRow] = typeArrayCopy[r][typeRow];
                        typeArrayCopy[r][typeRow] = t;
                    }
                    //elementary operations: switch XOR cols
                    for (int c = 0; c < n; c++) {
                        t = typeYihuoCopy[blockRow][c];
                        typeYihuoCopy[blockRow][c] = typeYihuoCopy[typeRow][c];
                        typeYihuoCopy[typeRow][c] = t;
                    }
                    //elementary operations: switch XOR rows
                    for (int r = 0; r < n; r++) {
                        t = typeYihuoCopy[r][blockRow];
                        typeYihuoCopy[r][blockRow] = typeYihuoCopy[r][typeRow];
                        typeYihuoCopy[r][typeRow] = t;
                    }
                    //elementary operations: switch XNOR cols
                    for (int c = 0; c < n; c++) {
                        t = typeTonghuoCopy[blockRow][c];
                        typeTonghuoCopy[blockRow][c] = typeTonghuoCopy[typeRow][c];
                        typeTonghuoCopy[typeRow][c] = t;
                    }
                    //elementary operations: switch XNOR rows
                    for (int r = 0; r < n; r++) {
                        t = typeTonghuoCopy[r][blockRow];
                        typeTonghuoCopy[r][blockRow] = typeTonghuoCopy[r][typeRow];
                        typeTonghuoCopy[r][typeRow] = t;
                    }
                }
                //if each matrix's rows matched
                if (colNum == n - 1) {
                    //recursively compare next rows until we find a isomorphic solution
                    Node[] result = getTransVertices(blockArray, blockYihuo, blockTonghuo, typeArrayCopy, typeYihuoCopy, typeTonghuoCopy, n, vertices, matchedRowNum, blockRow + 1);
                    if (result != null) {
                        return result;
                    }
                    if (matchedRowNum == n) {
                        //all rows match, but also the matrices been operated should be the same
                        if (isMaticxSame(blockArray, typeArray)) {
                            return vertices;
                        }
                        return null;
                    }
                    for (int i = 0; i < vertices.length; i++) {
                        vertices[i] = new Node(typeVertices[i].name, typeVertices[i].color);
                    }
                    matchedRowNum = matchedRowNum - 1;
                    //continue
                }
            }
        }
        return null;
    }
    private static boolean isMaticxSame(int[][] blockArray, int[][] typeArrayCopy) {
        if (blockArray.length != typeArrayCopy.length) {
            return false;
        }
        if (blockArray[0].length != typeArrayCopy[0].length) {
            return false;
        }
        for (int row = 0; row < blockArray.length; row++) {
            for (int col = 0; col < blockArray[0].length; col++) {
                if (blockArray[row][col] != typeArrayCopy[row][col]) {
                    return false;
                }
            }
        }
        return true;
    }
    /**
     * bubble sort
     */
    public static void bubbleSort(int mp[], int n) {
        int t;
        for (int i = 0; i < n - 1; i++) {
            for (int j = 0; j < n - 1 - i; j++) {
                if (mp[j] > mp[j + 1]) {
                    t = mp[j];
                    mp[j] = mp[j + 1];
                    mp[j + 1] = t;
                }
            }
        }
    }
    /**
     * calculate 01 adjacency matrix
     */
    public static void oneZero(int[][] p1, int[][] p2, int n) {
        for (int i = 0; i < n; i++) {
            for (int j = 0; j < n; j++) {
                if (p1[i][j] > 0) {
                    p2[i][j] = 1;
                } else {
                    p2[i][j] = 0;
                }
            }
        }
    }
    /**
     * calculate XOR matrix
     */
    public static void xor(int[][] p1, int[][] p2, int[][] p3, int n) {
        for (int i = 0; i < n; i++) {
            for (int j = 0; j < n; j++) {
                if (i == j) {
                    p3[i][j] = p1[i][i];
                } else {
                    int sum1, sum12;
                    sum1 = 0;
                    for (int k = 0; k < n; k++) {
                        if (p2[i][k] == p2[j][k]) {
                            sum12 = 0;
                        } else {
                            sum12 = 1;
                        }
                        sum1 = sum1 + (p1[i][k] + p1[j][k]) * sum12;
                    }
                    p3[i][j] = sum1;
                }
            }
        }
    }
    /**
     * calculate XNOR matrix
     */
    public static void xnor(int[][] p1, int[][] p2, int[][] p4, int n) {
        for (int i = 0; i < n; i++) {
            for (int j = 0; j < n; j++) {
                if (i == j) {
                    p4[i][j] = p1[i][i];
                } else {
                    int sum1, sum12;
                    sum1 = 0;
                    for (int k = 0; k < n; k++) {
                        if (p2[i][k] == p2[j][k]) {
                            sum12 = 1;
                        } else {
                            sum12 = 0;
                        }
                        sum1 = sum1 + (p1[i][k] + p1[j][k]) * sum12;
                    }
                    p4[i][j] = sum1;
                }
            }
        }
    }
}
------------------------------------------------
public class MvdColorMarker {
    public static List<MvdGraph> GRAPHS = new ArrayList<>();
    public int colorCount = 0;
    //load the template coloring graphs in the files
    static {
        GRAPHS.add(ReadGraphUtils.readFiles("graph_9Vertex-9.txt"));
        GRAPHS.add(ReadGraphUtils.readFiles("graph_9Vertex-11.txt"));
    }
    public void markBlocks(List<MvdGraph> blocks) {
        for (MvdGraph block : blocks) {
            findTypeAndColoring(block);
        }
        //Change the color of those vertices in the template block that have the same color as the cut-vertex to the color of the current cut-vertex
        for (MvdGraph block : blocks) {
            updateColorsOfCutVertex(block);
        }
    }
    private void updateColorsOfCutVertex(MvdGraph block) {
        if (block.template == null || block.vertices == null) {
            return;
        }
        for (int i = 0; i < block.vertices.length; i++) {
            Node node = block.vertices[i];
            if (!node.isCutVertex) {
                continue;
            }
            Node templateNode = block.template[i];
            for (int j = 0; j < block.template.length; j++) {
                if (block.template[j].color.equals(templateNode.color)) {
                    block.vertices[j].color = node.color;
                }
            }
        }
    }
    private void findTypeAndColoring(MvdGraph block) {
        System.out.println("### Isomorphic Relationship ###");
        for (MvdGraph type : GRAPHS) {
            findTypeAndColoring(block, type);
        }
    }
    private void findTypeAndColoring(MvdGraph graph, MvdGraph type) {
        Node[] template = IsomorphicJudger.getIsomorphicColors(graph, type);
        if (template == null) {
            return;
        }
        for (int i = 0; i < graph.vertices.length; i++) {
            //coloring the vertices
            graph.vertices[i].color = template[i].color + colorCount;
        }
        colorCount += graph.vertices.length;
        graph.template = template;
        printRelationships(graph,type,template);
    }
    private void printRelationships(MvdGraph graph, MvdGraph type,Node[] template) {
        System.out.println("current block:");
        PrintUtils.printVerticesArray(graph.vertices);
        System.out.println("isomorphic block before elementary operations:");
        PrintUtils.printVerticesArray(type.vertices);
        PrintUtils.printGraph(type.edges);
        System.out.println("isomorphic block: after elementary operations:");
        PrintUtils.printVerticesArray(template);
        System.out.println();
    }
}
------------------------------------------------
public class Node {
    public String name;
    public Node parent;
    public int depth;
    public int low;
    public Integer color;
    public Boolean isCutVertex = false;
    public Node(String name) {
        this.name = name;
        this.depth = 0;
        this.low = this.depth;
    }
    public Node(String name, Integer color) {
        this.name = name;
        this.color = color;
        this.depth = 0;
        this.low = this.depth;
    }
    public String toString() {
        if (color == null) {
            return "{" + "'" + name + '\'' + '}' ;
        }
        return "{" + "'" + name + '\'' + ":" + color + '}' + " ";
    }
    public boolean equals(Object o) {
        if (this == o) {
            return true;
        }
        if (o == null || getClass() != o.getClass()) {
            return false;
        }
        Node node = (Node) o;
        return depth == node.depth && low == node.low && color == node.color && Objects.equals(name, node.name) && Objects.equals(parent, node.parent);
    }
    public int hashCode() {
        return Objects.hash(name);
    }
}
------------------------------------------------
public class Edge {
    public int connected;
    public boolean reached;
    public Edge(int connected) {
        this.connected = connected;
        this.reached = false;
    }
}
------------------------------------------------
public class MvdGraph {
    public Edge[][] edges;
    public Node[] vertices;
    //record the colored vertices corresponding to vertices
    public Node[] template;
    public MvdGraph(Node[] vertices, Edge[][] edges) {
        this.vertices = vertices;
        this.edges = edges;
    }
    public void print() {
        for (Node node : vertices) {
            System.out.print(node);
        }
        System.out.println();
        printGraph(edges);
    }
    private void printGraph(Edge[][] edges) {
        for (int row = 0; row < edges.length; row++) {
            for (int col = 0; col < edges[0].length; col++) {
                System.out.printf("%d\t", edges[row][col].connected);
            }
            System.out.println();
        }
        System.out.println("----------------");
    }
}
------------------------------------------------
public class PrintUtils {
    public static void printGraph(Edge[][] edges) {
        for (int row = 0; row < edges.length; row++) {
            for (int col = 0; col < edges[0].length; col++) {
                System.out.printf("%d\t", edges[row][col].connected);
            }
            System.out.println();
        }
    }
    public static void printVerticesArray(Node[] nodes) {
        for (Node node:nodes) {
            System.out.print(node);
        }
        System.out.println();
    }
}
------------------------------------------------
public class ReadGraphUtils {
    public static MvdGraph readFiles(String fileName) {
        try {
            FileReader in = new FileReader("resources/"+fileName);
            BufferedReader reader = new BufferedReader(in);
            String line, verticesLines;
            ArrayList<ArrayList<Integer>> edgesList = new ArrayList<ArrayList<Integer>>();
            verticesLines = reader.readLine();
            Node[] vertices = createVertices(verticesLines.split(","));
            while ((line = reader.readLine()) != null) {
                String[] str = line.split(",");
                ArrayList lineEdge = new ArrayList<>();
                for (String edge : str) {
                    lineEdge.add(Integer.parseInt(edge.trim()));
                }
                edgesList.add(lineEdge);
            }
            Edge[][] edges = createEdges(edgesList);
            if (edges.length != vertices.length) {
                System.out.println("The number of vertices and edges should be equivalent");
                return null;
            }
            return new MvdGraph(vertices, edges);
        } catch (Exception e) {
            System.out.println(e);
            return null;
        }
    }
    private static Node[] createVertices(String[] strVertices) {
        Node[] vertices = new Node[strVertices.length];
        for (int i = 0; i < strVertices.length; i++) {
            vertices[i] = new Node(strVertices[i].split(":")[0].trim(), Integer.parseInt(strVertices[i].split(":")[1].trim()));
        }
        return vertices;
    }
    private static Edge[][] createEdges(ArrayList<ArrayList<Integer>> edgesList) {
        Edge[][] result = new Edge[edgesList.size()][edgesList.size()];
        for (int row = 0; row < result.length; row++) {
            for (int col = 0; col < result[0].length; col++) {
                result[row][col] = new Edge(edgesList.get(row).get(col));
            }
        }
        return result;
    }
}
------------------------------------------------
//resources directory
a:1, b:2, c:1, d:2, e:1, f:2, g:1, h:1, i:2
0, 1, 0, 0, 0, 1, 0, 0, 0
1, 0, 1, 0, 0, 0, 1, 0, 0
0, 1, 0, 1, 0, 0, 0, 1, 1
0, 0, 1, 0, 1, 0, 0, 0, 0
0, 0, 0, 1, 0, 1, 0, 0, 1
1, 0, 0, 0, 1, 0, 1, 1, 0
0, 1, 0, 0, 0, 1, 0, 0, 0
0, 0, 1, 0, 0, 1, 0, 0, 0
0, 0, 1, 0, 1, 0, 0, 0, 0

a:1, b:2, c:1, d:2, e:1, f:2, g:2, h:1, i:2
0, 1, 0, 0, 0, 1, 1, 0, 1
1, 0, 1, 0, 0, 0, 0, 0, 0
0, 1, 0, 1, 0, 0, 0, 0, 0
0, 0, 1, 0, 1, 0, 0, 1, 0
0, 0, 0, 1, 0, 1, 0, 0, 0
1, 0, 0, 0, 1, 0, 0, 0, 0
1, 0, 0, 0, 0, 0, 0, 1, 0
0, 0, 0, 1, 0, 0, 1, 0, 1
1, 0, 0, 0, 0, 0, 0, 1, 0
------------------------------------------------
//the output results
### Input Graph ###
0	0	0	0	1	0	0	0	0	0	0	0	0	0	0	1	0	
0	0	0	0	0	0	0	1	0	0	0	0	0	0	0	0	1	
0	0	0	0	0	0	0	0	0	0	0	1	0	0	1	0	0	
0	0	0	0	0	0	0	0	0	0	0	1	1	0	0	0	0	
1	0	0	0	0	1	0	0	0	0	1	0	0	1	0	0	0	
0	0	0	0	1	0	1	0	0	0	0	0	0	0	0	0	0	
0	0	0	0	0	1	0	1	0	1	0	0	0	0	0	1	0	
0	1	0	0	0	0	1	0	0	0	0	0	1	1	1	0	0	
0	0	0	0	0	0	0	0	0	0	0	1	1	0	0	0	0	
0	0	0	0	0	0	1	0	0	0	0	0	0	1	0	0	0	
0	0	0	0	1	0	0	0	0	0	0	0	0	0	0	1	0	
0	0	1	1	0	0	0	0	1	0	0	0	0	0	0	0	1	
0	0	0	1	0	0	0	1	1	0	0	0	0	0	0	0	0	
0	0	0	0	1	0	0	1	0	1	0	0	0	0	0	0	0	
0	0	1	0	0	0	0	1	0	0	0	0	0	0	0	0	0	
1	0	0	0	0	0	1	0	0	0	1	0	0	0	0	0	0	
0	1	0	0	0	0	0	0	0	0	0	1	0	0	0	0	0	
### CutVertices and Blocks  ###
cutVertices of Input Graph:
cutVerticesSet:[{'H'}]

Block generated from Graph:
Block num 1
[{'I'}, {'M'}, {'D'}, {'O'}, {'C'}, {'L'}, {'Q'}, {'B'}, {'H'}]
0	1	0	0	0	1	0	0	0	
1	0	1	0	0	0	0	0	1	
0	1	0	0	0	1	0	0	0	
0	0	0	0	1	0	0	0	1	
0	0	0	1	0	1	0	0	0	
1	0	1	0	1	0	1	0	0	
0	0	0	0	0	1	0	1	0	
0	0	0	0	0	0	1	0	1	
0	1	0	1	0	0	0	1	0	
Block num 2
[{'K'}, {'P'}, {'J'}, {'N'}, {'H'}, {'G'}, {'F'}, {'E'}, {'A'}]
0	1	0	0	0	0	0	1	0	
1	0	0	0	0	1	0	0	1	
0	0	0	1	0	1	0	0	0	
0	0	1	0	1	0	0	1	0	
0	0	0	1	0	1	0	0	0	
0	1	1	0	1	0	1	0	0	
0	0	0	0	0	1	0	1	0	
1	0	0	1	0	0	1	0	1	
0	1	0	0	0	0	0	1	0	

### Isomorphic Relationship ###
current block:
{'I':2} {'M':1} {'D':2} {'O':1} {'C':2} {'L':1} {'Q':2} {'B':1} {'H':2}
isomorphic block before elementary operations:
{'a':1} {'b':2} {'c':1} {'d':2} {'e':1} {'f':2} {'g':2} {'h':1} {'i':2}
0	1	0	0	0	1	1	0	1	
1	0	1	0	0	0	0	0	0	
0	1	0	1	0	0	0	0	0	
0	0	1	0	1	0	0	1	0	
0	0	0	1	0	1	0	0	0	
1	0	0	0	1	0	0	0	0	
1	0	0	0	0	0	0	1	0	
0	0	0	1	0	0	1	0	1	
1	0	0	0	0	0	0	1	0	
isomorphic block: after elementary operations:
{'g':2} {'h':1} {'i':2} {'e':1} {'f':2} {'a':1} {'b':2} {'c':1} {'d':2}

### Isomorphic Relationship ###
current block:
{'K':10} {'P':11} {'J':11} {'N':10} {'H':11} {'G':10} {'F':10} {'E':11} {'A':10}
isomorphic block before elementary operations:
{'a':1} {'b':2} {'c':1} {'d':2} {'e':1} {'f':2} {'g':1} {'h':1} {'i':2}
0	1	0	0	0	1	0	0	0	
1	0	1	0	0	0	1	0	0	
0	1	0	1	0	0	0	1	1	
0	0	1	0	1	0	0	0	0	
0	0	0	1	0	1	0	0	1	
1	0	0	0	1	0	1	1	0	
0	1	0	0	0	1	0	0	0	
0	0	1	0	0	1	0	0	0	
0	0	1	0	1	0	0	0	0	
isomorphic block: after elementary operations:
{'a':1} {'b':2} {'d':2} {'e':1} {'i':2} {'g':1} {'h':1} {'f':2} {'c':1}

### Coloring Vertices Results ###
{'A':10} {'B':1} {'C':11} {'D':11} {'E':11} {'F':10} {'G':10} {'H':11} {'I':11} {'J':11} {'K':10} {'L':1} {'M':1} {'N':10} {'O':1} {'P':11} {'Q':11}
Process finished with exit code 0
\end{lstlisting}
\end{appendix}

\begin{thebibliography}{1}
\bibitem{ref1}Bondy, J. A., Murty, U. S. R (1977). Graph theory with applications. Macmillan.
\bibitem{ref2}Chartrand, G., Johns, G. L., McKeon, K. A., Zhang, P. (2008). Rainbow connection in graphs. Mathematica Bohemica, 133(1), 85-98.
\bibitem{ref3}Caro, Y., Yuster, R. (2011). Colorful monochromatic connectivity. Discrete Mathematics, 311(16), 1786-1792.
\bibitem{ref4}Cai, Q., Li, X., Wu, D. (2015). Erd\H{o}s-Gallai-type results for colorful monochromatic connectivity of a graph. Journal of Combinatorial Optimization, 33(1), 123-131.
\bibitem{ref5}Cai, Q., Li, X., Wu, D. (2018). Some extremal results on the colorful monochromatic vertex-connectivity of a graph. Journal of Combinatorial Optimization, 35(4), 1300-1311.
\bibitem{ref6}Gu, R., Li, X., Qin, Z., Zhao, Y. (2015). More on the colorful monochromatic connectivity. Bulletin of the Malaysian Mathematical Sciences Society, 40(4), 1769-1779.
\bibitem{ref7}Hobbs, A. M. (1973). A catalog of minimal blocks. Journal of Research of the National Bureau of Standards, Section B: Mathematical Sciences, 77B, 53-60.
\bibitem{ref8}Jin, Z., Li, X., Wang, K. (2020). The monochromatic connectivity of graphs. Taiwanese Journal of Mathematics, 24(4), 785-815.
\bibitem{ref9}Krivelevich, M., Yuster, R. (2010). The rainbow connection of a graph is (at most) reciprocal to its minimum degree. Journal of Graph Theory, 63(3), 185-191.
\bibitem{ref10}Lu, Z. P., Ma, Y. B. (2014). Graphs with vertex rainbow connection number two. Science China Mathematics, 58(8), 1803-1810.
\bibitem{ref11}Li, X., Shi, Y., Sun, Y. (2013). Rainbow connections of Graphs: A survey. Graphs and Combinatorics, 29(1), 1-38.
\bibitem{ref12}Li, X., Sun, Y. (2017). An updated survey on rainbow connections of graphs-A dynamic survey. Theory and Applications of Graphs, 00(01), 1-65.
\bibitem{ref13}Li, X., Li, H., Ma, Y. (2021). The vertex-rainbow connection number of some graph operations. Discussiones Mathematicae Graph Theory, 41(2), 513-550.
\bibitem{ref14}Li, X., Wu, D. (2018). A survey on monochromatic connections of graphs. Theory and Applications of Graphs, 00(01), 1-19.
\bibitem{ref15}Li, X., Liu, S. (2013). A sharp upper bound for the rainbow 2-connection number of a 2-connected graph. Discrete Mathematics, 313(6), 755-759.
\bibitem{ref16}Li, P., Li, X. (2021). Monochromatic disconnection of graphs. Discrete Applied Mathematics, 288, 171-179.
\bibitem{ref17}Li, P., Li, X. (2021). Monochromatic disconnection: Erd\H{o}s-Gallai-type problems and product graphs. Journal of Combinatorial Optimization. https://doi.org/10.1007/s10878-021-00820-3
\bibitem{ref18}Li, P., Li, X. (2021). Upper Bounds for the MD-numbers and characterization of extremal graphs. Discrete Applied Mathematics, 295, 1-12.
\bibitem{ref20}Ma, Y., Chen, L., Li, H. (2017). Graphs with small total rainbow connection number. Frontiers of Mathematics in China, 12(4), 921-936.
\bibitem{ref21}Plummer, M. D. (1968). On minimal blocks. Transactions of the American Mathematical Society, 134(1), 85-94.
\bibitem{ref22}Tarjan, R. (1972). Depth-first search and linear graph algorithms. SIAM Journal on Computing, 1(2), 146-160.
\end{thebibliography}
\end{document}